\documentclass[a4paper,reqno]{amsart}
\usepackage{verbatim} 

\usepackage{tikz}
\usepackage{tikz-cd}

\pgfqkeys{/tikz/commutative diagrams}{%
  path/.style={/tikz/arrows=cm cap-cm cap}%
}

\usepackage{pstricks,pst-coil}
\usepackage{amsthm} 
\usepackage{graphics,graphicx}
\usepackage{stmaryrd}
\usepackage{latexsym}
\usepackage{amssymb}
\usepackage{amsmath}
\usepackage{extarrows}
\usepackage{upgreek}
\usepackage{mathrsfs} 
\newtheorem{theorem}{Theorem}[section]
\newtheorem{proposition}[theorem]{Proposition}
\newtheorem{lemma}[theorem]{Lemma}
\newtheorem{corollary}[theorem]{Corollary}

\theoremstyle{definition} 
\newtheorem{definition}[theorem]{Definition}
\newtheorem{remark}[theorem]{Remark}
\theoremstyle{remark} 
\newtheorem{example}[theorem]{Example}

\renewcommand{\and}{\quad\text{and}\quad}

\renewcommand{\for}{\text{ for }}

\newcommand{\ZZ}{\mathbb{Z}}

\newcommand{\NN}{\mathbb{N}}
\newcommand{\ra}{\rightarrow}
\newcommand{\lra}{\longrightarrow}

\newcommand{\cc}{{\mathcal C}}
\newcommand{\cd}{{\mathcal D}}

\newcommand{\cj}{{\mathcal J}}
\newcommand{\cp}{{\mathcal P}}
\newcommand{\ck}{{\mathcal K}}

\newcommand{\DR}{\mathop{{\operatorname{D}}_{\!R}}}
\newcommand{\DK}{\mathop{{\operatorname{D}}_{\!K}}}

\newcommand{\cpp}{{\mathcal P}(Q_\sigma,W_\sigma,F)}
\newcommand{\bas}{P_{Q_{\sigma}}/\hspace{-2pt}\sim}

\renewcommand{\mod}{\mathrm{mod}\,}
\DeclareMathOperator{\proj}{proj}
\DeclareMathOperator{\add}{add}

\DeclareMathOperator{\im}{Im}
\DeclareMathOperator{\tr}{Tr}
\DeclareMathOperator{\soc}{soc}
\DeclareMathOperator{\M}{M}
\DeclareMathOperator{\GL}{GL}
\DeclareMathOperator{\Id}{Id}

\DeclareMathOperator{\coker}{coker}
\DeclareMathOperator{\Hom}{Hom}
\DeclareMathOperator{\Ext}{Ext}
\DeclareMathOperator{\End}{End}
\DeclareMathOperator{\Aut}{Aut}

\DeclareMathOperator{\inj}{inj}
\DeclareMathOperator{\CM}{CM}
\DeclareMathOperator{\thick}{thick}
\DeclareMathOperator{\diag}{diag}
\DeclareMathOperator{\gldim}{gl. dim}
\DeclareMathOperator{\prdim}{proj. dim}
\DeclareMathOperator{\depth}{depth}
\DeclareMathOperator{\inte}{int}
\DeclareMathOperator{\Gr}{Gr}
\DeclareMathOperator*{\tens}{\otimes}

\newcommand{\shift}[2]{\ifstrequal{#2}{1}{\Sigma #1}{\Sigma^{#2}#1}}

\newcommand{\vvect}[1]{
  \begin{bmatrix}
    #1
  \end{bmatrix}
  }
\newcommand{\hvect}[1]{
  \begin{bmatrix}
    #1
  \end{bmatrix}
  }
  
    \def\restr#1#2{\mathchoice
                  {\setbox1\hbox{${\displaystyle #1}_{\scriptstyle #2}$}
                  \restrictionaux{#1}{#2}}
                  {\setbox1\hbox{${\textstyle #1}_{\scriptstyle #2}$}
                  \restrictionaux{#1}{#2}}
                  {\setbox1\hbox{${\scriptstyle #1}_{\scriptscriptstyle #2}$}
                  \restrictionaux{#1}{#2}}
                  {\setbox1\hbox{${\scriptscriptstyle #1}_{\scriptscriptstyle #2}$}
                  \restrictionaux{#1}{#2}}}
    \def\restrictionaux#1#2{{#1\,\smash{\vrule height .8\ht1 depth .85\dp1}}_{\,#2}}

\newcommand{\vech}[2]{
    \begin{bmatrix}
      #1 & #2
    \end{bmatrix}
    }
\newcommand{\mat}[4]{
    \begin{bmatrix}
      #1 & #2 \\
      #3 & #4
    \end{bmatrix}
    }

\usepackage[normalem]{ulem}
\usepackage{cancel}
\usepackage{soul}

\definecolor{orange}{rgb}{1,0.90,0.5}
\let\oldmarginpar\marginpar
\renewcommand\marginpar[1]{(\hl{**})\oldmarginpar[\raggedleft\footnotesize \fcolorbox{blue}{orange}{\parbox{\marginparwidth}{\color{blue}{(**) #1}}}]%
{\raggedright\footnotesize \fcolorbox{blue}{orange}{\parbox{\marginparwidth}{\color{blue}{(**) #1}}}}}

\newcommand{\rep}[2]{{\color{red}\ifmmode\ifx&#1&\else\cancel{#1}\ \fi\else\ifx&#1&\else\sout{#1}\ \fi\fi\color{green!50!black}#2}}

\numberwithin{equation}{section}
\title[Ice quivers with potential associated with triangulations of polygons]{Ice quivers with potential associated with triangulations and Cohen-Macaulay modules over orders}
\author[L. Demonet]{Laurent Demonet}
\author[X. Luo]{Xueyu Luo}


\begin{document}

\begin{abstract}
Given a triangulation of a polygon $P$ with $n$ vertices, we associate an ice quiver with potential such that the associated Jacobian algebra has the structure of a Gorenstein tiled $K[x]$-order $\Lambda$.
Then we show that the stable category of the category of Cohen-Macaulay $\Lambda$-modules is equivalent to the cluster category $\cc$ of type $A_{n-3}$.
 It gives a natural interpretation of the usual indexation of cluster tilting objects of $\cc$ by triangulations of $P$. Moreover, it extends naturally the triangulated categorification by $\cc$ of the cluster algebra of type $A_{n-3}$ to an exact categorification by adding coefficients corresponding to the sides of $P$. Finally, we lift the previous equivalence of categories to an equivalence between the stable category of graded Cohen-Macaulay $\Lambda$-modules and the bounded derived category of modules over a path algebra of type $A_{n-3}$. 
\end{abstract}

\maketitle

\section{Introduction}
\label{s:Introduction2}

In 2001, Fomin and Zelevinsky introduced a new class of algebras with a rich combinatorial structure, called \emph{cluster algebras} \cite{FoZe02}, \cite{FoZe03} motivated by canonical bases and total positivity \cite{Ka91}, \cite{Lu91}, \cite{Lu97}. 
Since their emergence, cluster algebras have generated a lot of interest, coming in particular from their links with many other subjects: combinatorics, Poisson geometry, integrable systems, Teichm\"uller spaces, and, last but not least, representations of finite-dimensional algebras.

In seminal articles, Marsh, Reineke, Zelevinsky \cite{MaReZe03}, Buan, Marsh, Reineke, Reiten, Todorov \cite{BMRRT} and Caldero, Chapoton \cite{CaCh06} have shown that the important class of acyclic cluster algebras could be modeled with categories constructed from representations of quivers. Under these categorifications, clusters are represented by objects having a strong homological property, called \emph{cluster tilting objects}. Since that time, these categories, called cluster categories, were widely generalized in several directions.
One of these generalizations, was given by Amiot \cite{Afourier} by using quivers with potential, introduced by Derksen, Weyman and Zelevinsky in \cite{DWZ1}. A quiver with potential is a quiver $Q$ together with a linear combination $W$ of cyclic paths in $Q$. To each quiver with potential is associated an important algebra, $\cp(Q,W)$, called the \emph{Jacobian algebra}. 

For each triangulation of a bordered surface with marked points,  a quiver was introduced by Caldero, Chapoton and Schiffler in type $A$ \cite{CPF06} and Fomin, Shapiro and Thurston in general cases \cite{FST08}. 
They showed that the combinatorics of triangulations of the surface correspond to that of the cluster algebra defined by the quiver.
Later in \cite{LF09}, Labardini-Fragoso associated to each triangulation $\sigma$ of a bordered surface with marked points a potential $W(\sigma)$ on the corresponding quiver $Q(\sigma)$. 
He proved that flips of triangulations are compatible with mutations of quivers with potential.
This is partially generalized to the case of tagged triangulations by Labardini-Fragoso and Cerulli Irelli in \cite{LF11}, and completely by Labardini-Fragoso in \cite{LF12}.

The aim of this paper is to enhance some of these known results by considering ice quivers with potential and certain Frobenius categories given by Cohen-Macaulay modules. 
The study of Cohen-Macaulay modules (or lattices) over orders is a classical subject in representation theory. We refer to \cite{FMO,RT,POS,CM} for a general background on this subject and also to \cite{AIRCC,Araya,Thanvanden,IyTa13,Matrixrepresentation,Matrixweight} for recent results about connections with tilting theory. Recently a strong connection between Cohen-Macaulay representation theory and cluster categories has been found \cite{AIRCC,Thanvanden,IyTa13,KR}.
We will enlarge this connection by studying the frozen Jacobian algebras associated with triangulations of surfaces from the viewpoint of Cohen-Macaulay representation theory.
Through this paper, let $K$ denote a field and $R=K[x]$.  We extend the construction of \cite{CPF06}, and associate an ice quiver with potential $(Q_\sigma, W_\sigma)$ to each triangulation $\sigma$ of a polygon $P$ with $n$ vertices by adding a set $F$ of $n$ frozen vertices corresponding to the edges of the polygon and certain arrows (see Definition \ref{def:QP}). We study the associated frozen Jacobian algebra 
\[  \Gamma_\sigma:=\cpp \]
(see Definition \ref{frozen JA}). Our first main results are the following:
\begin{theorem}
[Theorem \ref{thm:frozen Jacobian algebra is order} and \ref{thm:frozen part}]
Let $e_F$ be the sum of the idempotents of $\Gamma_\sigma$ at frozen vertices. Then 
\begin{enumerate}
\item the frozen Jacobian algebra $\Gamma_\sigma$ has the structure of an $R$-order (see Definition \ref{def:order} and Remark \ref{rem:order}).
\item the $R$-order $e_F \Gamma_\sigma e_F$ is isomorphic to the Gorenstein tiled $R$-order 
\begin{equation}
    \label{the order}
  \Lambda:=  \begin{bmatrix}
    R & R& R&\cdots&R & (x^{-1})\\
    (x) & R& R&\cdots& R&R\\
    (x^2) &(x)& R&\cdots& R & R\\
    \vdots&\vdots&\vdots&\ddots&\vdots&\vdots\\
    (x^2) &(x^2)& (x^2)&\cdots& R & R\\
    (x^2) & (x^2)& (x^2)&\cdots  &(x) &R
   \end{bmatrix}_{n\times n.}
\end{equation}
\end{enumerate}
\end{theorem} 

\begin{remark}
 The order $\Lambda$ is an almost Bass order of type (IVa), see for example \cite{Iya-order}.
\end{remark}

\begin{theorem}[Theorems \ref{thm:frozen Jacobian algebra is order}, \ref{thm:frozen part}, \ref{thm:edges and modules}, \ref{thm:cluster tilting} and \ref{th:cyeq}]
 \begin{enumerate}
  \item For any triangulation $\sigma$ of $P$, we can map each edge $i$ of $\sigma$ to the indecomposable Cohen-Macaulay $\Lambda$-module $e_F \Gamma_\sigma e_i$ where $e_i$ is the idempotent of $\Gamma_\sigma$ corresponding to $i$. In fact, this module does only depend on $i$ and this construction induces one-to-one correspondences 
 \begin{align*}
  \left\{\text{sides and diagonals of $P$} \right\} & \longleftrightarrow \left\{\text{indecomposable objects of $\CM(\Lambda)$}\right\}/\cong \\
  \left\{\text{sides of $P$} \right\} & \longleftrightarrow \left\{\text{indecomposable projectives of $\CM(\Lambda)$}\right\}/\cong\\
  \left\{\text{triangulations of $P$} \right\} & \longleftrightarrow \left\{\text{basic cluster tilting objects of $\CM(\Lambda)$}\right\}/\cong.
 \end{align*} 
  \item For the basic cluster tilting object $T_\sigma$ corresponding to a triangulation $\sigma$, $$\End_{\CM(\Lambda)}(T_\sigma) \cong \Gamma_\sigma^{\operatorname{op}}.$$
  \item The category $\underline{\CM}(\Lambda)$ is $2$-Calabi-Yau.
  \item If $K$ is a perfect field, there is a triangle-equivalence $\cc(K Q) \cong \underline{\CM}(\Lambda)$ where $Q$ is a quiver of type $A_{n-3}$ and $\cc(K Q)$ is the corresponding cluster category. 
 \end{enumerate}
\end{theorem}

Note that this theorem permits to enhance the usual cluster category of type $A_{n-3}$ with projective objects corresponding through the categorification to coefficients of the cluster algebra, in a way which fits perfectly with the combinatorial interpretation of triangulations of polygons. Indeed, there is exactly one coefficient per side of the polygon and the situation is invariant by rotation of the polygon. More precisely, using the cluster character $X'$ defined in \cite[section 3]{FuKeller} on Frobenius categories (see also \cite{Palu}), we get the following theorem.
\begin{theorem}[Theorem \ref{th:grassman}]
 If $K$ is algebraically closed, the category $\CM(\Lambda)$ categorifies through the cluster character $X'$ a cluster algebra structure on the homogeneous coordinate ring of the Grassmannian of $2$-dimensional planes in $L^n$ (for any field $L$). This cluster algebra structure coincides with the one defined in \cite{FoZe02} from Pl\"ucker coordinates.
\end{theorem}

Similar results are obtained in the categorifications of the cluster structures of coordinate rings of (general) Grassmannians independently by  Baur, King, Marsh \cite{BKM} (see also Jensen, King, Su \cite[Theorem 4.5]{JKS}).

Usually, the cluster category $\cc(K Q)$ is constructed as an orbit category of the bounded derived category $\cd^{\operatorname{b}}(K Q)$. We can reinterpret this result in this context by studying the category of Cohen-Macaulay graded $\Lambda$-modules $\CM^\ZZ(\Lambda)$:

\begin{theorem}[Theorem \ref{thm:cluster cat as stable cat for A type}]
  Using the same notation as before: 
  \begin{enumerate}
   \item The Cohen-Macaulay $\Lambda$-module $e_F \Gamma_\sigma$ can be lifted to a tilting object in $\underline{\CM}^\ZZ(\Lambda)$.
   \item There exists a triangle-equivalence $\cd^{\operatorname{b}}(K Q) \cong \underline{\CM}^{\ZZ}(\Lambda)$.
  \end{enumerate}
\end{theorem}


\medskip\noindent{\bf Acknowledgments }
We would like to show our gratitude to the second author's supervisor, Osamu Iyama, for his guidance and valuable discussions. We also thank him for introducing us this problem and showing us enlightening examples.
We would also like to thank Erik Darp\"o, Martin Herschend, Gustavo Jasso and Dong Yang for their helpful comments and advice. 

\section{Ice Quivers with Potentials associated with triangulations}
\label{s:Ice Quivers with Potentials associated with Triangulations}
In this section, we introduce ice quivers with potential associated with triangulations and their frozen Jacobian algebras. We show that in our case, the frozen Jacobian algebras have a structure of $R$-orders, and the frozen parts of the Jacobian algebras are isomorphic to $\Lambda$ defined in \eqref{the order} as an $R$-order.

\subsection{Frozen Jacobian algebras}
\label{ss:Frozen Jacobian algebras}
We refer to \cite{DWZ1} for background about quivers with potential.
Let $Q$ be a finite connected quiver without loops, with set of vertices $Q_0=\{1,\ldots,n\}$ and set of arrows $Q_1$. 
As usual, if $a \in Q_1$, we denote by $s(a)$ its starting vertex and by $e(a)$ its ending vertex.
We denote by $KQ_i$ the $K$-vector space with basis $Q_i$ consisting of paths of length $i$ in $Q$, and by $KQ_{i,{\rm cyc}}$ the subspace of $KQ_i$ spanned by all cycles in $KQ_i$. 
Consider the path algebra $KQ=\bigoplus_{i\ge0}KQ_i$. 
An element $W\in \bigoplus_{i\ge1}KQ_{i,{\rm cyc}}$ is called a {\em potential}.
Two potentials $W$ and $W'$ are called \emph{cyclically equivalent} if $W-W'$ belongs to $[KQ,KQ]$, the vector space spanned by commutators.
A \emph{quiver with potential} is a pair $(Q,W)$ consisting of a quiver $Q$ without loops and a potential $W$ which does not have two cyclically equivalent terms. 
For each arrow $a\in Q_1$, the cyclic derivative $\partial_a$ is the $K$-linear map $\bigoplus_{i\ge1}KQ_{i,{\rm cyc}}\to KQ$ defined on cycles by
\[
\partial_a(a_1\cdots a_d)=\sum_{a_i=a}a_{i+1}\cdots a_d a_1\cdots a_{i-1}.\]




\begin{definition} 
\label{frozen JA}
\cite{BIRS}
An \emph{ice quiver with potential} is a triple $(Q,W,F)$, where $(Q,W)$ is a quiver with potential and $F$ is a subset of $Q_0$.
Vertices in $F$ are called \emph{frozen vertices}.
   The \emph{frozen Jacobian algebra} is defined by \[\cp(Q,W,F)=KQ/\cj(W,F),\] where $\cj(W,F)$ is the ideal \[\cj(W,F)=\langle \partial_a W \mid a \in Q_1, \ s(a)\notin F\ \mbox{ or }\ e(a)\notin F \rangle\] of $KQ$. 
\end{definition}

\begin{figure}[t]
 \[{ \begin{tikzcd}[ampersand replacement=\&]
         \&             \&               \& 1\drar{b_1}                 \&                \&                 \&\\
         \&             \& 2\drar{b_2}\urar{a_1}     \&                      \& 3\arrow{ll}[swap]{c_1} \drar{b_3} \&    \&\\
       \& 4\urar{a_2}   \&        \& 5\arrow{ll}[swap]{c_2}\urar{a_3} \&   \& 6\arrow{ll}[swap]{c_3} \&\\ 
  \end{tikzcd}}
  \]
\caption{Example of iced quiver with potential}
\label{ausl}
\end{figure}
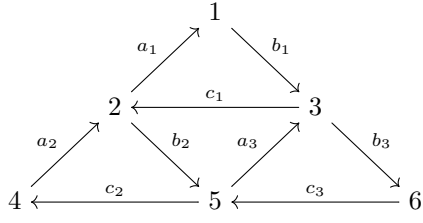

\begin{example}
Consider the quiver $Q$ of Figure \ref{ausl}
 with potential $W=a_1b_1c_1+a_2b_2c_2+a_3b_3c_3-c_1b_2a_3$
  and set of frozen vertices $F=\{4,5,6\}.$
  Then the Jacobian ideal is 
  \begin{align*}
  \cj(W,F) =\langle b_1c_1,c_1a_1,a_1b_1-b_2a_3,b_2c_2,c_2a_2-a_3c_1,b_3c_3-c_1b_2,c_3a_3  \rangle.
  \end{align*}
\end{example}

Note that this ice quiver with potential can be constructed from preprojective algebras \cite{BIRS,GLS11}.
\subsection{Ice quivers with potential arising from triangulations}
\label{ss:quiver with potential arising from triangulations}
We recall the definition of triangulations of polygons and introduce our definition of ice quivers with potential arising from triangulations of a polygon.
\begin{definition}
\label{def:QP}
Let $P$ be a regular polygon with $n$ vertices and $n$ sides.
A diagonal of $P$ is a line segment connecting two vertices of $P$ and lying in the interior of $P$.
Two diagonals are \emph{crossing} if they share one interior point, and they are \emph{non-crossing} if they share no interior points.
A \emph{triangulation} of $P$ is a decomposition of $P$ into triangles by a maximal set of non-crossing diagonals.

Fix a triangulation $\sigma$ of the polygon $P$.
We assign to $\sigma$ an ice quiver with potential $(Q_\sigma, W_\sigma, F)$ as follows.

The quiver $Q_{\sigma}$ of the triangulation $\sigma$ is a quiver whose vertices are indexed by all internal and external edges of $\sigma$ (the external edges are the sides of $P$ and the internal edges are the diagonals of $\sigma$). 
   Whenever two edges $a$ and $b$ are sides of a common triangle of $\sigma$, then $Q_{\sigma}$ contains an \emph{internal arrow} $a \ra b$ in the triangle if $a$ is a predecessor of $b$ with respect to anti-clockwise orientation centered at the joint vertex.
  Moreover, for every vertex of the polygon with at least one incident internal edge in $\sigma$, there is an \emph{external arrow} $a \ra b$ where $a$ and $b$ are its two incident external edges, $a$ being a predecessor of $b$ with respect to anti-clockwise orientation centered at the joint vertex.
 
An \emph{internal path} is a path consisting of internal arrows. A \emph{minimal cycle} of $Q_\sigma$ is a cycle in which no arrow appears more than once, and which encloses a part of the plane whose interior is connected and does not contain any arrow of $Q_\sigma$. For example, in Figure \ref{pentagon}, $\alpha b e h$ is a minimal cycle, but $\alpha b c \beta \gamma g h$ and $c \beta d e h \alpha b$ are not. There are two kinds of minimal cycles of $Q_\sigma$: \emph{cyclic triangles}, which consist of three internal arrows inside a triangle of $\sigma$, and \emph{big cycles}, which consist of internal arrows and one external arrow around a vertex of $P$. 
An \emph{almost complete big cycle} means a path in $Q_{\sigma}$ which can be completed into a big cycle by adding an internal arrow.

We define $F$ as the subset of $(Q_{\sigma})_0$ indexed by the $n$ sides of the polygon, and the potential $W_{\sigma}$ as $$\sum \text{cyclic\ triangles} - \sum \text{big\ cycles}.$$ 
\end{definition}

We name the vertices of $P$ by $P_1,P_2,\ldots,P_n$ counterclockwise.
We number $(Q_\sigma)_0$ from $1$ to $2n-3$ in such a way that $(P_i,P_{i+1})$ is numbered by $i$ for $1\le i\le n-1$ and $(P_n,P_1)$ is numbered by $n$. 
From now on, for any triangulation $\sigma$ of the polygon $P$, denote $\cpp$ by $\Gamma_\sigma$.

\begin{example}
 We illustrate the construction of $Q_{\sigma}$ and $W_{\sigma}$ if $\sigma$ is the triangulation of the pentagon in Figure \ref{pentagon}.
   In this case $W_{\sigma}=abc+def+ghi-\alpha beh-\beta dc-\gamma gf$, $F=\{1,2,3,4,5\}$ and
  $$\cj(W_{\sigma},F) = \langle ca - eh\alpha, ab - \beta d, ef-c\beta, fd-h\alpha b, de-\gamma g, ig-\alpha be, hi-f\gamma \rangle.$$
   \begin{figure}[t]
\scalebox{1} 
{
\begin{pspicture}(0,-1.7986894)(5.3229103,1.8306894)
\definecolor{color639}{rgb}{1.0,0.0,0.2}
\definecolor{color648}{rgb}{0.0,0.2,1.0}
\pspolygon[linewidth=0.016](3.9541228,0.7813978)(2.5330968,1.7512873)(1.133119,0.7476629)(1.7856674,-0.9728546)(3.5044217,-0.8961507)(3.516309,-0.8704857)
\psline[linewidth=0.016cm](3.9165704,0.76762015)(1.1587839,0.73577553)
\psline[linewidth=0.016cm](3.9234593,0.748844)(1.7975547,-0.9471896)
\psline[linewidth=0.016cm,linecolor=color639,arrowsize=0.07cm 10.0,arrowlength=2.0,arrowinset=0.6]{->}(1.94671,1.2591994)(3.0965087,1.3188903)
\psline[linewidth=0.016cm,linecolor=color639,arrowsize=0.07cm 10.0,arrowlength=2.0,arrowinset=0.6]{->}(2.5051234,0.7823613)(1.8784943,1.212868)
\psline[linewidth=0.016cm,linecolor=color639,arrowsize=0.07cm 10.0,arrowlength=2.0,arrowinset=0.6]{->}(3.0658453,1.2863365)(2.5495646,0.7773628)
\psline[linewidth=0.016cm,linecolor=color639,arrowsize=0.07cm 10.0,arrowlength=2.0,arrowinset=0.6]{->}(1.4618925,-0.09037531)(2.4951265,0.69347906)
\psline[linewidth=0.016cm,linecolor=color639,arrowsize=0.07cm 8.0,arrowlength=2.0,arrowinset=0.6]{->}(2.5583436,0.69536936)(2.8288984,-0.10011799)
\psline[linewidth=0.016cm,linecolor=color639,arrowsize=0.07cm 10.0,arrowlength=2.0,arrowinset=0.6]{->}(2.7844572,-0.09511949)(1.482559,-0.14670378)
\psline[linewidth=0.016cm,linecolor=color639,arrowsize=0.07cm 10.0,arrowlength=2.0,arrowinset=0.6]{->}(2.9177806,-0.110114984)(3.725828,-0.04798836)
\psline[linewidth=0.016cm,linecolor=color639,arrowsize=0.07cm 10.0,arrowlength=2.0,arrowinset=0.6]{->}(3.725828,-0.04798836)(2.682597,-0.920725)
\psline[linewidth=0.016cm,linecolor=color639,arrowsize=0.07cm 10.0,arrowlength=2.0,arrowinset=0.6]{->}(2.6569319,-0.9088377)(2.8307886,-0.16333528)
\psbezier[linewidth=0.016,linecolor=color648,arrowsize=0.07cm 10.0,arrowlength=2.0,arrowinset=0.6]{->}(3.754441,-0.080097646)(4.8932076,0.08206386)(4.4008427,1.7122262)(3.115285,1.3257792)
\psbezier[linewidth=0.016,linecolor=color648,arrowsize=0.07cm 10.0,arrowlength=2.0,arrowinset=0.6]{->}(1.8321627,1.2810837)(1.0227705,1.6871432)(0.19363868,0.63731444)(1.431229,-0.122929126)
\psbezier[linewidth=0.016,linecolor=color648,arrowsize=0.07cm 10.0,arrowlength=2.0,arrowinset=0.6]{->}(1.4465011,-0.089506)(0.8933879,-1.0565202)(2.262574,-1.7906893)(2.676205,-0.963523)
\usefont{T1}{ptm}{m}{n}
\rput(3.1724708,1.637291){$1$}
\usefont{T1}{ptm}{m}{n}
\rput(1.7324708,1.577291){$2$}
\usefont{T1}{ptm}{m}{n}
\rput{-3.8304546}(0.017683469,0.064378574){\rput(0.9524707,-0.24270904){$3$}}
\usefont{T1}{ptm}{m}{n}
\rput{-0.36567238}(0.0075401687,0.018429808){\rput(2.8724706,-1.1827091){$4$}}
\usefont{T1}{ptm}{m}{n}
\rput(4.0924706,-0.26270905){$5$}

\usefont{T1}{ptm}{m}{n}
\rput(2.5500000,0.797291){$6$}
\usefont{T1}{ptm}{m}{n}
\rput(2.9000000,-0.15270905){$7$}

\usefont{T1}{ptm}{m}{n}
\rput(2.4424708,1.4172909){$a$}
\usefont{T1}{ptm}{m}{n}
\rput(2.9924707,0.99729097){$b$}
\usefont{T1}{ptm}{m}{n}
\rput{3.6363735}(0.067826286,-0.121106036){\rput(1.9224707,0.99729097){$c$}}
\usefont{T1}{ptm}{m}{n}
\rput(1.8524708,0.35729095){$d$}
\usefont{T1}{ptm}{m}{n}
\rput{-3.2419322}(-0.01844946,0.16360523){\rput(2.8624706,0.39729095){$e$}}
\usefont{T1}{ptm}{m}{n}
\rput{-1.9640993}(0.010576258,0.072549306){\rput(2.1024706,-0.28270903){$f$}}
\usefont{T1}{ptm}{m}{n}
\rput(2.5524707,-0.60270905){$g$}
\usefont{T1}{ptm}{m}{n}
\rput(3.3524706,0.07729095){$h$}
\usefont{T1}{ptm}{m}{n}
\rput(3.3124707,-0.5827091){$i$}
\usefont{T1}{ptm}{m}{n}
\rput(4.662471,0.797291){$\alpha$}
\usefont{T1}{ptm}{m}{n}
\rput(0.4924707,0.77729094){$\beta$}
\usefont{T1}{ptm}{m}{n}
\rput(1.2824707,-1.2027091){$\gamma$}
\end{pspicture} 
}
   \caption{}
   \label{pentagon}
 \end{figure}
\end{example}

\begin{remark}
  Note that $\partial_a W_{\sigma}$ is not in $\cj(W_{\sigma},F)$ as both the source and the sink of $a$ are in $F$. All arrows that are not linking two vertices in $F$ are shared by two cycles in $Q_{\sigma}$, one of which is a cyclic triangle and the other is a big cycle. Thus all relations in $\Gamma_\sigma$ are commutativity relations. 
\end{remark}

\subsection{Frozen Jacobian algebras are $R$-orders}
\label{ss:Jacobian Alg as Order}
Let $(Q_{\sigma},W_{\sigma},F)$ be an ice quiver with potential arising from a triangulation $\sigma$ as defined in Section \ref{ss:quiver with potential arising from triangulations} and $e_i$ be the trivial path of length $0$ at vertex $i$.
The main result of this section is that $\Gamma_\sigma$ is an $R$-order.

First, we introduce the definition of orders and Cohen-Macaulay modules over orders.
\begin{definition} 
\label{def:order}
Let $S$ be a commutative Noetherian ring of Krull dimension 1.
An $S$-algebra $A$ is called an {\emph{$S$-order}} if it is a finitely generated $S$-module and $\soc_S A=0$. 
 For an $S$-order $A$, a left $A$-module $M$ is called a (maximal) {\emph{Cohen-Macaulay $A$-module}} if it is finitely generated as an $S$-module and $\soc_S M=0$ (or equivalently $\soc_A M=0$). 
  We denote by $\CM(A)$ the category of Cohen-Macaulay $A$-modules. It is a full exact subcategory of $\mod A$.
\end{definition}  

\begin{remark}
\label{rem:order}
If $S$ is a principal ideal domain (\emph{e.g.} $S=R$) and $M \in \mod S$, then $\soc_S M=0$ if and only if $M$ is free as $S$-module.
\end{remark}

We refer to \cite{FMO},\cite{RT},\cite{POS} and \cite{CM} for more details about orders and Cohen-Macaulay modules.

\begin{definition}
\label{def:tiled order}
An $R$-order $A$ is called a \emph{tiled $R$-order} if $A=((x^{a_{ij}}))_{i,j\in\{1,\ldots,n\}} \subset \M_n(R)$ for $a_{i,j}\in \ZZ$ satisfying $a_{ij}+a_{jt}\ge a_{it}$ for all $i,j,t\in \{1,2,\ldots,n\}${ where $(x^{a_{i,j}})$ is the $R$-submodule of the fraction field of $R$ generated by $x^{a_{i,j}}$.}
\end{definition}

Then the main theorem of this subsection is the following.
\begin{theorem}
\label{thm:frozen Jacobian algebra is order}
The frozen Jacobian algebra $\Gamma_\sigma$ has the structure of a tiled $R$-order.
\end{theorem}

We will give a proof of Theorem \ref{thm:frozen Jacobian algebra is order} in the rest of this subsection. The strategy is to construct a central element $C$ in $\Gamma_\sigma$ such that for each pair of vertices $(i,j)$ of $Q_\sigma$, $e_i\Gamma_\sigma e_j$ is a free module of rank 1 over $K[C]$.
We start by defining the central element $C$.
\begin{definition} We say that two paths $w_1$ and $w_2$ of $Q_\sigma$ are \emph{equivalent} if $w_1=w_2$ holds in $\Gamma_\sigma$. 
In this case, we write $w_1\sim w_2$.
\end{definition}

The basis of the path algebra $KQ_{\sigma}$ consisting of all paths in $Q_{\sigma}$ is denoted by $P_{Q_{\sigma}}$. 

\begin{lemma} 
\label{lem:basis of frozen algebra}
 The set $\bas$ is a basis of the frozen Jacobian algebra $\Gamma_\sigma$. Moreover, $e_i \Gamma_\sigma e_j \cap (\bas)$ is a basis of $e_i \Gamma_\sigma e_j$.
\end{lemma}
\begin{proof}
 Recall that all relations here are commutativity relations. 
\end{proof}

Let $i\in (Q_{\sigma})_0$. All minimal cycles of $Q_\sigma$ passing through $i$ are equal in $\Gamma_\sigma$. We denote the common element that they represent in $\Gamma_\sigma$ by $C_i$.
We now define $C:=\sum_{i\in (Q_\sigma)_0}C_i \in \Gamma_\sigma$. 
We get the following proposition.
\begin{proposition}
  The element $C$ is central in $\Gamma_\sigma$.
\end{proposition}
\begin{proof}It suffices to show that for each arrow $a\in (Q_{\sigma})_1$, the equation $Ca=aC$ holds.
  Indeed, for each arrow $a:\ i\ra j$, $C_i=a p$ and $C_j=p a$ holds for some path $p$ from $j$ to $i$. Thus we have $$Ca=\sum_{k\in Q_0}C_ka=\sum_{k\in Q_0}C_k e_i a=C_i a=a p a=a C_j=a e_j \sum_{k\in Q_0}C_k=aC.$$
\end{proof}

\begin{example}
Let us calculate $C$ for the triangulation $\sigma$ of the pentagon in Figure \ref{pentagon}.
In $\Gamma_\sigma$, we have $C_1=bca=beh\alpha,\ C_2=abc=\beta dc,\ C_3=dc\beta=def=\gamma gf,\ C_4=ghi=gf\gamma,\ C_5=igh=\alpha beh,\ C_6=cab=c\beta d=efd=eh\alpha b$ and  $C_7=hig=f\gamma g=fde=h\alpha be$. 
So $C=bca+abc+dc\beta +ghi+igh+cab +hig.$
\end{example}

Next, we define a grading on $KQ_\sigma$ such that $\cj(W_\sigma,F)$ is a homogeneous ideal. Each angle of the polygon is $(n-2)\pi/n$. The angles of triangles of $\sigma$ are multiples of $\pi/n$. 
\begin{definition}
\label{def:length}
The \emph{$\theta$-length} of an internal arrow in $Q_{\sigma}$ from $i$ to $j$ is $t$ $(1\leq t\leq n-2)$ if the angle between the two edges $i$ and $j$ of the triangle in $\sigma$ is $t\pi/n$.
The \emph{$\theta$-length} of an external arrow in $Q_{\sigma}$ is 2.

This extends additively to a map $\ell^{\theta}$ from paths in $Q_\sigma$ to integers, defining a grading on $KQ_{\sigma}$ which will also be denoted by $\ell^{\theta}$.
\end{definition}

\begin{example}
Consider the triangulation $\sigma$ of the pentagon in Figure \ref{pentagon}.
By calculating the angles of the triangulation, we can get the $\theta$-length of each arrow in $Q_{\sigma}$:
$\ell^{\theta}(a)= 3,\ \ell^{\theta}(b)= 1,\ \ell^{\theta}(c)= 1,\ \ell^{\theta}(d)= 2,\ \ell^{\theta}(e)= 1,\ \ell^{\theta}(f)= 2,\ \ell^{\theta}(g)= 1,\ \ell^{\theta}(h)= 1,\ \ell^{\theta}(i)= 3,\ \ell^{\theta}(\alpha)= 2,\ \ell^{\theta}(\beta)= 2,\ \ell^{\theta}(\gamma)= 2.$
From this, the minimal $\theta$-length of paths between any two vertices can be calculated.

For all paths $w_{1i}$ from $1$ to $i$ ($1 \le i \le 7$) with minimal $\theta$-length, we have:
$w_{11} \sim e_1$, $w_{12} \sim  bc$, $w_{13} \sim  bc\beta$, $w_{14} \sim  bc\beta\gamma$, $w_{15} \sim  beh$, $w_{16} \sim b$ and $w_{17} \sim  be$. We will prove that the paths with minimal $\theta$-length are the generators of $\Gamma_\sigma$ as $R$-order.
\end{example}

The following facts about $\ell^\theta$ are readily derived from the definition:
\begin{lemma}
 \begin{enumerate}
  \item The potential $W_{\sigma}$ is homogeneous of $\theta$-length $n$.
  
    \item If $w_1 \sim w_2$, then $\ell^{\theta}(w_1)=\ell^{\theta}(w_2)$ holds and $\ell^{\theta}$ induces the structure of a graded algebra on $\Gamma_\sigma$.
  
  \item The element $C\in \Gamma_\sigma$ is homogeneous of $\theta$-length $n$.

 \item An internal arrow between frozen vertices always has $\theta$-length $n-2$.
 \end{enumerate}
\end{lemma}

In order to study the structure of $\Gamma_\sigma$, we prove the following lemma by induction.
\begin{theorem}
\label{th:induction}
 \begin{enumerate}
    \item Let $i,j\in  (Q_\sigma)_0$. There exists a path $w_0(i,j)$ from $i$ to $j$ with minimal $\theta$-length such that for any path $w$ from $i$ to $j$, $w\sim w_0(i,j)C^N$ holds for some non-negative integer $N$.

    \item Let $\ell^{\theta}_{i,j}=\ell^\theta(w_0(i,j))$ for any $i,j\in (Q_\sigma)_0$. Then
      \begin{equation}
       \ell^{\theta}_{i,j}
       = \max_{k\in F} \{ \ell^{\theta}_{k,i} -\ell^{\theta}_{k,j} \}.
    \end{equation}
 \end{enumerate}
\end{theorem}
Before proving Theorem \ref{th:induction}, let us introduce another description of the frozen Jacobian algebra $\cpp$.
For each internal arrow $a \in (Q_{\sigma})_1$ between two frozen vertices $l+1$ and $l$, we add an external arrow $a^*:l\ra l+1$ to obtain an extended version $Q'_{\sigma}$ of $Q_\sigma$.
In such a case, we call the triangle $P_lP_{l+1}P_{l+2}$ a \emph{corner triangle}. 

We define the \emph{internal Jacobian algebra} $\cp_{\inte}(Q'_{\sigma},W'_{\sigma})$ by 
\[\cp_{\inte}(Q'_{\sigma},W'_{\sigma})=KQ'_{\sigma}/\cj_{\inte}(W'_{\sigma}),\]
 where $W'_{\sigma}:=W_\sigma - \sum_{a^*\in (Q'_{\sigma})_1\smallsetminus(Q_{\sigma})_1} aa^*$ and $\cj_{\inte}(W'_{\sigma}) :=\langle \partial_a W'_{\sigma} \mid \text{$a \in (Q'_\sigma)_1$ is internal} \rangle.$ 
 
 It is easy to check that we have an isomorphism 
 \begin{align*}
 \cp_{\inte}(Q'_{\sigma},W'_{\sigma}) &\cong \cp(Q_\sigma,W_\sigma,F)\\
 (Q_\sigma)_1 \ni a &\mapsto a\\
 (Q'_{\sigma})_1\smallsetminus(Q_{\sigma})_1 \ni  a* & \mapsto \partial_a W_\sigma.
 \end{align*} 
 
 Therefore, the two algebras have the same $K$-basis and the central element $C$ of $\cp(Q_\sigma,W_\sigma,F)$ can be regarded as a central element $C$ of $\cp_{\inte}(Q'_{\sigma},W'_{\sigma})$. We will prove Theorem~\ref{th:induction} for $\cp_{\inte}(Q'_{\sigma},W'_{\sigma})$.

\begin{proof}[Proof of Theorem \ref{th:induction}]
We will prove (1) and (2) by induction over $n$. When $n=3$, (1) and (2) are obvious.

Suppose that $n\ge4$ and suppose that (1) and (2) are proved for $n-1$.
 Let $i,j\in (Q'_\sigma)_0$. There exists a corner triangle $P_lP_{l+1}P_{l+2}$ such that $j\notin \{l, l+1\}$ since $n\ge 4$ implies that $\sigma$ has at least two corner triangles. 
We denote by $\tau$ the restriction of $\sigma$ to the polygon obtained by removing $P_{l+1}$, $l$ and $l+1$.
Let $\cp_{\inte}(Q'_\tau, W'_\tau)$ be the internal Jacobian algebra associated to $\tau$ and $C_\tau$ be its central element.
See Figure \ref{triangulation}.
\begin{figure}[t]
\scalebox{1} 
{
\begin{pspicture}(0,-2.48)(8.881894,2.48)
\definecolor{color42}{rgb}{1.0,0.0,0.2}
\pstriangle[linewidth=0.016,dimen=outer](4.22,-0.3)(5.4,1.68)
\psline[linewidth=0.016cm](1.52,-0.3)(0.44,-2.46)
\psline[linewidth=0.016cm](6.92,-0.3)(8.0,-2.46)
\usefont{T1}{ptm}{m}{n}
\rput(4.181455,1.55){$P_{l+1}$}
\usefont{T1}{ptm}{m}{n}
\rput(7.3114552,-0.135){$P_{l}$}
\usefont{T1}{ptm}{m}{n}
\rput(1.3714551,-0.035){$P_{l+2}$}
\psline[linewidth=0.016cm,linecolor=red,arrowsize=0.07cm 8.0,arrowlength=2.0,arrowinset=0.6]{->}(3.2,0.68)(5.24,0.66)
\psline[linewidth=0.016cm,linecolor=red,arrowsize=0.07cm 8.0,arrowlength=2.0,arrowinset=0.6]{->}(5.26,0.66)(4.28,-0.26)
\psline[linewidth=0.016cm,linecolor=red,arrowsize=0.07cm 8.0,arrowlength=2.0,arrowinset=0.6]{->}(4.22,-0.3)(3.22,0.64)
\usefont{T1}{ptm}{m}{n}
\rput(5.571455,1.065){$l$}
\usefont{T1}{ptm}{m}{n}
\rput(2.861455,1.205){$l+1$}
\usefont{T1}{ptm}{m}{n}
\rput(4.361455,0.505){$a$}
\usefont{T1}{ptm}{m}{n}
\rput(4.6,0.225){$b$}
\usefont{T1}{ptm}{m}{n}
\rput(3.900000,0.225){$c$}
\psbezier[linewidth=0.016,linecolor=blue,arrowsize=0.07cm 8.0,arrowlength=2.0,arrowinset=0.6]{->}(5.3,0.68)(5.34,2.28)(3.24,2.46)(3.22,0.74)
\usefont{T1}{ptm}{m}{n}
\rput(8.081455,-1.595){$l-1$}
\usefont{T1}{ptm}{m}{n}
\rput(0.4214551,-1.495){$l+2$}
\psbezier[linewidth=0.016,linecolor=blue,arrowsize=0.07cm 8.0,arrowlength=2.0,arrowinset=0.6]{->}(3.18,0.76)(1.1945455,2.02)(0.0,0.31333333)(0.94,-1.34)
\psbezier[linewidth=0.016,linecolor=blue,arrowsize=0.07cm 8.0,arrowlength=2.0,arrowinset=0.6]{->}(7.54,-1.48)(8.56,-0.35757226)(7.6512055,1.76)(5.36,0.68)
\usefont{T1}{ptm}{m}{n}
\rput(8.221455,0.505){$\alpha$}
\usefont{T1}{ptm}{m}{n}
\rput(4.131455,2.185){$a^*$}
\usefont{T1}{ptm}{m}{n}
\rput(0.89145505,1.105){$\beta$}
\usefont{T1}{ptm}{m}{n}
\rput(4.181455,-0.455){$m$}
\usefont{T1}{ptm}{m}{n}
\rput(4.571455,-1.815){$\tau$}
\end{pspicture} 
} 
\scalebox{1} 
{
\begin{pspicture}(0,-2.21)(8.36,2.21)
\psline[linewidth=0.016cm](2.34,0.21)(5.54,0.21)
\psline[linewidth=0.016cm](2.34,0.21)(0.74,-2.19)
\psline[linewidth=0.016cm](5.54,0.21)(7.14,-2.19)
\usefont{T1}{ptm}{m}{n}
\rput(5.871455,0.315){$P_{l}$}
\usefont{T1}{ptm}{m}{n}
\rput(1.911455,0.295){$P_{l+2}$}
\usefont{T1}{ptm}{m}{n}
\rput(6.801455,-1.045){$l-1$}
\usefont{T1}{ptm}{m}{n}
\rput(3.841455,-0.085){$m$}
\usefont{T1}{ptm}{m}{n}
\rput(0.6814551,-1.545){$l+2$}
\usefont{T1}{ptm}{m}{n}
\rput(4.211455,-1.285){$\tau$}
\psbezier[linewidth=0.016,linecolor=blue,arrowsize=0.07cm 8.0,arrowlength=2.0,arrowinset=0.6]{->}(6.34,-0.99)(8.34,1.01)(5.1844444,1.81)(3.94,0.21)
\psbezier[linewidth=0.016,linecolor=blue,arrowsize=0.07cm 8.0,arrowlength=2.0,arrowinset=0.6]{->}(3.86,0.25)(2.7,2.19)(0.0,0.51)(1.16,-1.51)
\usefont{T1}{ptm}{m}{n}
\rput(7.3114552,0.575){$\alpha'$}
\usefont{T1}{ptm}{m}{n}
\rput(2.241455,1.295){$\beta'$}
\end{pspicture} 
}
     \caption{Triangulations $\sigma$ and $\tau$}
   \label{triangulation}
\end{figure}
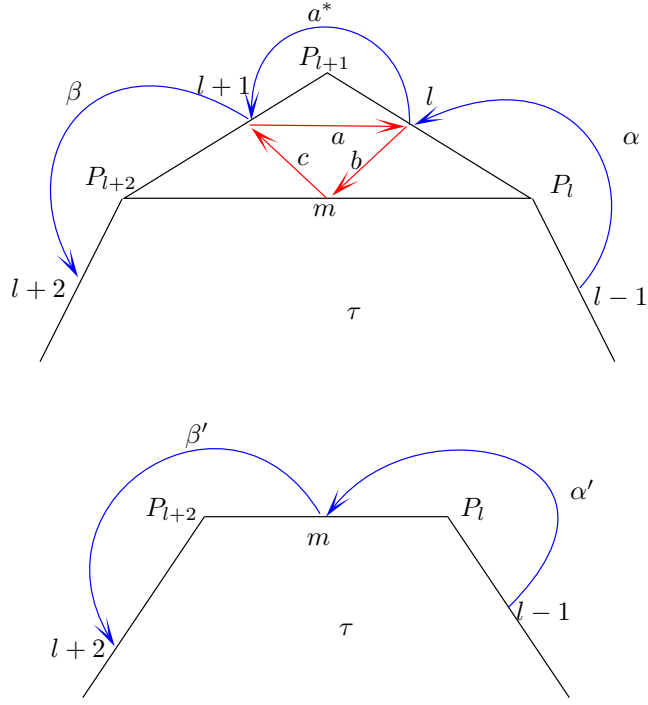

Then there is a non-unital monomorphism 
\begin{align*}
\xi: KQ'_{\tau} & \hookrightarrow KQ'_{\sigma}\\
\alpha' &\mapsto \alpha b,\\
\beta'  &\mapsto c\beta,\\
(Q'_{\tau})_1\smallsetminus\{\alpha', \beta'\} \ni \gamma  &\mapsto \gamma .
\end{align*}
Let $w$ be a path from $i$ to $j$ in $Q'_\sigma$. Up to relations, we can suppose that:
\begin{itemize}
\item $w$ does not contain $a^*$. As $a^* \sim bc$, we can replace $a^*$ by $bc$;
\item $w$ does not contain $a$. Indeed, since $j\neq l$ holds, it follows that $w$ contains $a$ if and only if $w$ contains $ab$. Up to the relation $\partial_c W'_\sigma$, the path $ab$ is equivalent to a path which does not contain $a$.
\end{itemize}
 
First, suppose that $i \notin \{l,l+1\}$.
By induction hypothesis, there exists a path $v_0(i,j)$ with minimal $\theta$-length in $KQ'_{\tau}$ such that for any path $v$ from $i$ to $j$ in $Q'_\tau$, $v \sim v_0(i,j) C_{\tau}^N$ holds for some non-negative integer $N$.
Let us show that $w_0(i,j):=\xi(v_0(i,j))$ satisfies the desired property. 
Let $w$ be a path in $Q'_\sigma$ from $i$ to $j$. Using the previous observation, we can assume that $w$ does not contain $a$ or $a^*$ up to $\sim$. Thus there is a path $v$ of $Q'_\tau$ from $i$ to $j$ such that $w = \xi(v)$. Moreover, there is a non-negative integer $N$ such that $v \sim v_0(i,j) C_\tau^N$.  
Since $\xi$ sends relations to relations, it follows that $w \sim w_0(i,j) \xi(C_{\tau})^N$ holds.
Moreover $$\xi(C_{\tau})= C- C_l -C_{l+1} \and w_0(i,j)C_l=w_0(i,j)C_{l+1}=0$$ hold. Hence $w\sim w_0(i,j) C^N$.

Next, suppose that $i=l$ and put $w_0(l,j) :=b w_0(m,j)$. For any path $w$ from $l$ to $j$, $w\sim bp$ for some path $p$ from $m$ to $j$.
We saw that $p\sim w_0(m,j)C^N$ for some non-negative integer $N$, so $w\sim bw_0(m,j)C^N \sim w_0(l,j)C^N$.

Finally, suppose that $i=l+1$. Then one can show that $w_0(l+1,j)=\beta w_0(l+2,j)$ satisfies the desired property by a similar argument.

Clearly, the path $w_0(i,j)$ has minimal $\theta$-length. Indeed, for any path $w$ in $KQ'_\sigma$ from $i$ to $j$, 
$w \sim w_0(i,j) C^N$ holds for some non-negative integer $N$. It follows that $$\ell^\theta(w) = \ell^\theta(w_0(i,j)) + \ell^\theta(C^N)  \ge \ell^\theta(w_0(i,j)).$$ 
It finishes the proof of (1) for $n$ vertices.

Let us prove (2) for $n$ vertices. Since $ \ell^{\theta}_{k,i}+\ell^{\theta}_{i,j} \ge \ell^{\theta}_{k,j}$ holds for any $k \in F$,  it is sufficient to prove that there exists a frozen vertex $k$ such that $w_0(k,j) \sim w_0(k,i)w_0(i,j)$. 

If $i\in F$, then $k=i$ works.
Otherwise, $i \notin \{l,l+1\}$ and therefore by induction hypothesis, there exists a frozen vertex $k'$ in the triangulation $\tau$ such that $v_0(k',j) \sim v_0(k',i)v_0(i,j)$ holds. 
\begin{itemize}
  \item If $k\neq m$, then put $k=k'$;
 \item If $k=m$, then put $k=l$.
 \end{itemize}
In both cases, we have $w_0(k,j) \sim w_0(k,i)w_0(i,j)$.
\end{proof}

\begin{remark}
The non-unital monomorphism $\xi$ induces an isomorphism $$(1-e_{l+1})\cp_{\inte}(Q'_\sigma, W'_\sigma) (1-e_{l+1}) \cong \cp_{\inte}(Q'_\tau, W'_\tau).$$
\end{remark}

According to the above discussion, we get the following corollary.
\begin{corollary}
  \label{thm:basis of Jacobian algebra}
  For any two vertices $i,j\in (Q_{\sigma})_0$,
  the map 
  \begin{align*}
  \NN_{\ge0} &\to e_i \Gamma_\sigma e_j \cap (\bas)  \\ 
   l &\mapsto w_0(i,j)C^l
   \end{align*}
  is a bijection, where $w_0(i,j)$ is the path from $i$ to $j$ constructed in Theorem \ref{th:induction}.
\end{corollary}
\begin{proof}
It is surjective by Theorem \ref{th:induction}~(1).
It is injective since we have the grading $\ell^\theta$ such that $\ell^\theta(C)=n$.
\end{proof}
 
 We have the following immediate consequences.
\begin{remark}
 Let $i,j\in (Q_{\sigma})_0$ and $w$ be a path from $i$ to $j$. 
 \begin{itemize}
    \item The path $w$ has minimal $\theta$-length if and only if there is no other path $w'$ from $i$ to $j$ such that $w \sim w'C$ holds.
    \item If $\ell^\theta(w)<n$, then $w$ is a path with minimal $\theta$-length. 
  \end{itemize}  
\end{remark}

We consider the $K$-linear map  
\begin{align}
\label{equ:mor}
\phi: R &\hookrightarrow \Gamma_\sigma \\
x^i &\mapsto C^i  \ \ \ (i\ge0). \notag
\end{align}
 Then $\phi$ is a ring homomorphism and $\Gamma_\sigma$ is an $R$-algebra. Now, we are ready to prove Theorem \ref{thm:frozen Jacobian algebra is order}.

\begin{proof}[Proof of Theorem \ref{thm:frozen Jacobian algebra is order}]
According to Lemma \ref{lem:basis of frozen algebra} and Corollary \ref{thm:basis of Jacobian algebra}, the frozen Jacobian algebra $\Gamma_\sigma$ has a $K$-basis $\bas$ which consists of $ C^k w_0(i,j) $ for $k \in \NN_{\ge0}$ and $i,j\in (Q_\sigma)_0$. 
Thus, for $i,j \in (Q_\sigma)_0$, there is an isomorphism of $R$-modules:
\begin{align*}
R &\to e_i\Gamma_\sigma e_j  \\
r &\mapsto \phi(r)w_0(i,j).
\end{align*}
It follows that $\Gamma_\sigma= \bigoplus_{i,j\in(Q_\sigma)_0}e_i \Gamma_\sigma e_j$ is a finitely generated free $R$-module.
Consequently, the frozen Jacobian algebra $\Gamma_\sigma$ is an $R$-order.
\end{proof}

\subsection{Combinatorial structures of frozen Jacobian algebras}
\label{ss: Combinatorial Structure of Jacobian algebras}
In this section, we give an explicit description of $\Gamma_\sigma$ as a tiled order. Moreover, we prove that the frozen part $e_F \Gamma_\sigma e_F$ of $\Gamma_\sigma$ is isomorphic to $\Lambda$ given in \eqref{the order} for any triangulation $\sigma$, where $e_F$ is the sum of the {idempotents} at all frozen vertices of ${Q_\sigma}$.

We show that the minimal $\theta$-length between two vertices does not depend on the triangulation by the following proposition. 
\begin{proposition}
\label{pro:independence}
Let $\sigma$ and $\sigma'$ be two different triangulations of the polygon $P$. For any two edges $(P_i,P_{i'}),(P_j,P_{j'})$ common to $\sigma$ and $\sigma'$, the minimal $\theta$-length of paths from $(P_i,P_{i'})$ to $(P_j,P_{j'})$ in $Q_\sigma$ is the same as the one in $Q_{\sigma'}$. 
\end{proposition}
\begin{proof}
Consider the modified versions $(Q'_\sigma,W'_\sigma,F)$ and $(Q'_{\sigma'},W'_{\sigma'},F)$. The two triangulations are related by a sequence of flips which do not modify $(P_i, P_{i'})$ and $(P_j, P_{j'})$, so, without losing generality, we can assume that the only difference between $\sigma$ and $\sigma'$ is shown in Figure \ref{trisigsigp} and the other parts of $\sigma$ and $\sigma'$ are equal. 
Hence it is sufficient to prove that the minimal $\theta$-lengths of paths between two different external vertices in the diagrams of Figure \ref{trisigsigp} are the same.

\begin{figure}[t]
 \begin{pspicture}(0,-2)(5.5,2)
\psframe[linewidth=0.016,dimen=outer](4.08,1.2)(1.58,-1.3)
\psline[linewidth=0.016cm](1.6,1.16)(4.06,-1.28)
\psline[linewidth=0.016cm,linecolor=red
,arrowsize=0.07cm 8.0,arrowlength=2.0,arrowinset=0.6]{->}(4.08,-0.1)(2.88,-0.12)
\psline[linewidth=0.016cm,linecolor=red
,arrowsize=0.07cm 8.0,arrowlength=2.0,arrowinset=0.6]{->}(2.96,1.18)(4.0,-0.02)
\psline[linewidth=0.016cm,linecolor=red
,arrowsize=0.07cm 8.0,arrowlength=2.0,arrowinset=0.6]{->}(2.9,-0.1)(2.94,1.14)
\psline[linewidth=0.016cm,linecolor=red
,arrowsize=0.07cm 8.0,arrowlength=2.0,arrowinset=0.6]{->}(2.9,-0.2)(2.9,-1.24)
\psline[linewidth=0.016cm,linecolor=red
,arrowsize=0.07cm 8.0,arrowlength=2.0,arrowinset=0.6]{->}(1.62,-0.12)(2.86,-0.14)
\psline[linewidth=0.016cm,linecolor=red
,arrowsize=0.07cm 8.0,arrowlength=2.0,arrowinset=0.6]{->}(2.92,-1.26)(1.66,-0.14)
\usefont{T1}{ptm}{m}{n}
\rput(4.7,-0.1){$(P_k,P_l)$}
\usefont{T1}{ptm}{m}{n}
\rput(1.0,-0.1){$(P_i,P_{j})$}
\usefont{T1}{ptm}{m}{n}
\rput(2.9,1.45){$(P_l,P_i)$}
\usefont{T1}{ptm}{m}{n}
\rput(2.9,-1.6){$(P_j,P_k)$}
\usefont{T1}{ptm}{m}{n}
\rput(2.1,-0.8){$\alpha$}
\usefont{T1}{ptm}{m}{n}
\rput(2.25,0.1){$\beta$}
\usefont{T1}{ptm}{m}{n}
\rput(3.1,-0.75){$\gamma$}
\usefont{T1}{ptm}{m}{n}
\rput(3.8,0.6){$\alpha'$}
\usefont{T1}{ptm}{m}{n}
\rput(3.6,-0.35){$\beta'$}
\usefont{T1}{ptm}{m}{n}
\rput(2.7,0.5){$\gamma'$}
\end{pspicture}  \hspace{1cm} \scalebox{1} 
{
\begin{pspicture}(0,-2)(5.5,2)
\psframe[linewidth=0.015,dimen=outer](4.08,1.2)(1.58,-1.3)
\psline[linewidth=0.015cm](1.58,-1.3)(4.08,1.2)
\psline[linewidth=0.015cm,linecolor=red,arrowsize=0.07cm 8.0,arrowlength=2.0,arrowinset=0.6]{<-}(2.8,-0.08)(2.8,1.15)
\psline[linewidth=0.015cm,linecolor=red,arrowsize=0.07cm 8.0,arrowlength=2.0,arrowinset=0.6]{<-}(4.08,-0.1)(2.88,-0.12)
\psline[linewidth=0.015cm,linecolor=red,arrowsize=0.07cm 8.0,arrowlength=2.0,arrowinset=0.6]{<-}(1.62,-0.12)(2.7,-0.12)
\psline[linewidth=0.015cm,linecolor=red,arrowsize=0.07cm 8.0,arrowlength=2.0,arrowinset=0.6]{<-}(2.8,-1.25)(4.08,-0.12)
\psline[linewidth=0.015cm,linecolor=red,arrowsize=0.07cm 8.0,arrowlength=2.0,arrowinset=0.6]{->}(2.8,-1.25)(2.8,-0.12)
\psline[linewidth=0.015cm,linecolor=red,arrowsize=0.07cm 8.0,arrowlength=2.0,arrowinset=0.6]{<-}(2.8,1.15)(1.62,-0.12)

\usefont{T1}{ptm}{m}{n}
\rput(1.0,-0.1){$(P_i,P_j)$}
\usefont{T1}{ptm}{m}{n}
\rput(4.7,-0.1){$(P_k,P_l)$}
\usefont{T1}{ptm}{m}{n}
\rput(2.9,1.45){$(P_l,P_i)$}
\usefont{T1}{ptm}{m}{n}
\rput(2.9,-1.6){$(P_j,P_k)$}
\usefont{T1}{ptm}{m}{n}
\rput(4.7,-0.1){$(P_k,P_l)$}
\usefont{T1}{ptm}{m}{n}
\rput(1.0,-0.1){$(P_i,P_{j})$}
\usefont{T1}{ptm}{m}{n}
\rput(2.9,1.45){$(P_l,P_i)$}
\usefont{T1}{ptm}{m}{n}
\rput(2.9,-1.6){$(P_j,P_k)$}
\usefont{T1}{ptm}{m}{n}
\rput(2.0,0.7){$\tilde \alpha$}
\usefont{T1}{ptm}{m}{n}
\rput(2.2,-0.4){$\tilde \beta$}
\usefont{T1}{ptm}{m}{n}
\rput(2.6,-0.75){$\tilde \gamma'$}
\usefont{T1}{ptm}{m}{n}
\rput(3.7,-0.7){$\tilde \alpha'$}
\usefont{T1}{ptm}{m}{n}
\rput(3.5,0.1){$\tilde \beta'$}
\usefont{T1}{ptm}{m}{n}
\rput(3.0,0.5){$\tilde \gamma$}

\end{pspicture} 
}
 \caption{Triangulations $\sigma$ and $\sigma'$}
 \label{trisigsigp}
\end{figure}
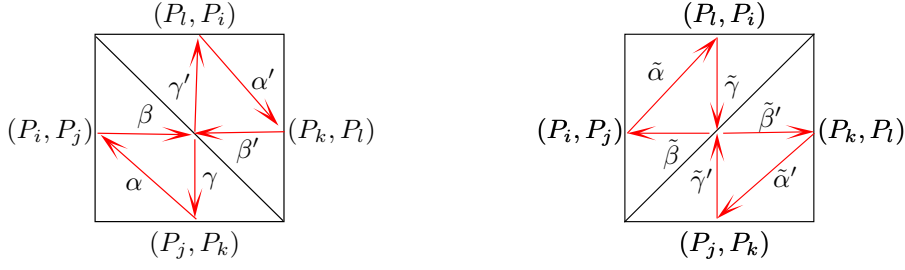

The reasoning can then be done case by case. By symmetry, we can consider minimal $\theta$-length of paths starting from $(P_i, P_j)$:
\begin{itemize}
 \item From $(P_i, P_j)$ to $(P_j, P_k)$, as $\ell^\theta(\beta \gamma) < n$, $\beta \gamma$ is of minimal $\theta$-length. Using the relation $\partial_\alpha W'_\sigma$, it is equivalent to a path $w$ which still exists in $\sigma'$ and therefore the minimal $\theta$-length of paths from $(P_i, P_j)$ to $(P_j, P_k)$ is the same in $\sigma'$;
 \item From $(P_i, P_j)$ to $(P_l, P_i)$, as $\ell^\theta(\beta \gamma') < n$, $\beta \gamma'$ is of minimal $\theta$-length. Moreover, $\ell^\theta(\tilde \alpha) = \ell^\theta(\beta \gamma')$;
 \item From $(P_i, P_j)$ to $(P_k, P_l)$, by Corollary \ref{thm:basis of Jacobian algebra}, $\beta \gamma' \alpha' \sim w_0((P_i,P_j), (P_k, P_l)) C^l$ for some $l \in \NN$ and $\beta' \gamma \alpha \sim w_0((P_k, P_l), (P_i, P_j)) C^{l'}$, so 
 $$e_{(P_i, P_j)}C^2 \sim w_0((P_i,P_j), (P_k, P_l)) w_0((P_k, P_l), (P_i, P_j)) C^{l+l'}.$$
 Moreover, $(P_i, P_j)$ and $(P_k, P_l)$ are not in the same minimal cycle (as they are not incident to the same vertex of $P$), so $l+l' = 0$ and $\beta \gamma' \alpha'$ is of minimal $\theta$-length. 
 Finally, $\ell^\theta(\beta \gamma' \alpha') = \ell^\theta(\tilde \alpha \tilde \gamma \tilde \beta')$. 
\end{itemize}
\end{proof}

Roughly speaking, the minimal $\theta$-length of paths between two vertices of the quiver $Q_\sigma$ corresponding to two edges of the triangulation reflects the angle we need to rotate one of the edges to make it coincide with the other one. However, this is true only modulo $n$. Unfortunately, this description is not in general enough to compute the minimal $\theta$-length which can actually be bigger than $n$. Its explicit value is given in the following proposition.
\begin{proposition}
\label{pro:length}
Let $(P_i,P_j)$ and $(P_k,P_l)$ be two vertices in $Q_\sigma$. The minimal $\theta$-length of paths from $(P_i,P_j)$ to $(P_k,P_l)$ is $$k+l-i-j+n\min\{\delta_{j>k}+{\delta_{i>l}}, {\delta_{i>k}}+{\delta_{j>l}}\},$$
where ${\delta_{i>j}}=
{\begin{cases}
1, & \text{if } i > j, \\
0, & \text{if } i \le j.
\end{cases}}$
\end{proposition}
\begin{proof}
If $\{i,j\} =\{k,l\}$ holds, then $(P_i,P_j)$ and $(P_k,P_l)$ are the same vertex. The minimal $\theta$-length is $0$.

Otherwise, if $\{i,j\} \cap \{k,l\} \neq \varnothing$ holds, assume that $i=k$ holds. 
It is easy to check that the minimal $\theta$-length from $(P_i,P_j)$ to $(P_i,P_l)$ is $l-j+n {\delta_{j>l}}$. Indeed, thanks to Proposition \ref{pro:independence}, we can choose the triangulation in which we do the computation. Thus, choose the one with edges $(P_i, P_{i'})$ for all $i'$ not equal or adjacent to $i$.

If $i,j,k,l$ are distinct, let us prove that there exists a path from $(P_i,P_j)$ to $(P_k,P_l)$ with $\theta$-length $$k+l-i-j+n\min\{{\delta_{j>k}}+{\delta_{i>l}},{\delta_{i>k}}+{\delta_{j>l}}\}.$$ 
By Proposition \ref{pro:independence}, the minimal $\theta$-length of paths from $(P_i,P_j)$ to $(P_k,P_l)$ does not depend on the triangulation. We can assume the triangulation contains the edges $(P_i,P_j)$, $(P_k,P_l)$, $(P_i,P_k)$ and $(P_i,P_l)$.
Consider the paths $$w_0((P_i,P_j),(P_i,P_l))w_0((P_i,P_l),(P_k,P_l))$$ and $$w_0((P_i,P_j),(P_i,P_k))w_0((P_i,P_k),(P_k,P_l))$$ from $(P_i,P_j)$ to $(P_k,P_l)$.
Since $$\ell^\theta(w_0((P_i,P_j),(P_i,P_l)))=l-j+n{\delta_{j>l}},$$ $$\ell^\theta(w_0((P_i,P_l),(P_k,P_l)))=k-i+n{\delta_{i>k}},$$ $$\ell^\theta(w_0((P_i,P_j),(P_i,P_k)))=k-j+n{\delta_{j>k}}$$ and $$\ell^\theta(w_0((P_i,P_k),(P_k,P_l)))=l-i+n{\delta_{i>l}},$$ it follows that there exists a path of $\theta$-length $$k+l-i-j+n\min\{{\delta_{j>k}}+{\delta_{i>l}},{\delta_{i>k}}+{\delta_{j>l}}\}$$ from $(P_i,P_j)$ to $(P_k,P_l)$.

Let us now prove by induction on the $\theta$-length of a path $w$ from $(P_i,P_j)$ to $(P_k,P_l)$ that this $\theta$-length is at least $$k+l-i-j+n\min\{{\delta_{j > k}}+{\delta_{i > l}}, {\delta_{i > k}}+{\delta_{j > l}}\}.$$
When $\{i,j\} =\{k,l\}$, this is clear.
When $\{i,j\} \neq \{k,l\}$, $w$ is the composition of a path $w'$ from $(P_i,P_j)$ to $(P_i,P_{j'})$ and a path $w''$ from $(P_i,P_{j'})$ to $(P_k,P_l)$ for some $j'{\neq j}$ (or similarly, it can be realized as the composition of a path $w'$ from $(P_i,P_j)$ to $(P_{i'},P_j)$ and a path $w''$ from $(P_{i'},P_j)$ to $(P_k,P_l)$ for some $i'{\neq i}$).
By the induction hypothesis, we have 
\begin{align*}
\ell^{\theta}(w)&= \ell^{\theta}(w')+\ell^{\theta}(w'') \\
&\ge  j'-j+n\delta_{j>j'}+k+l-i-j'+n\min\{\delta_{j'>k}+{\delta_{i > l}}, {\delta_{i > k}}+\delta_{j'>l}\} \\
&=  k+l-i-j+n\min\{\delta_{j'>k}+{\delta_{i > l}}+\delta_{j>j'},{\delta_{i > k}}+\delta_{j'>l}+\delta_{j>j'}\}\\
&\ge  k+l-i-j+n\min\{{\delta_{j > k}}+{\delta_{i > l}},{\delta_{i > k}}+{\delta_{j > l}}\}.
\end{align*}
\end{proof}

In the rest of this section, we will give an explicit description of the frozen Jacobian algebra $\Gamma_\sigma$ as a tiled $R$-order given by $\theta$-length.

For $i,j,k\in (Q_\sigma)_0$, there exists a non-negative integer $d^j_{i,k}$ such that $w_0(i,j)w_0(j,k) \sim C^{d^j_{i,k}}w_0(i,k)$. Then we have
\begin{equation}
\label{eq:length}
\ell^{\theta}_{i,j}/n+\ell^{\theta}_{j,k}/n-\ell^{\theta}_{i,k}/n=d^j_{i,k}{\in \NN}.
\end{equation}

Using the ring homomorphism  
\begin{align*}
     R=K[x] &\to K[t]\\
     x &\mapsto t^n,
\end{align*}
we regard $R$ as a subring of $K[t]$.
We identify $(Q_\sigma)_0$ with $\{1,2, \ldots, 2n-3\}$ and $F$ with $\{1,2, \ldots, n\}$.
Consider the $R$-module defined by $${\Gamma'_\sigma} := (t^{\ell^{\theta}_{i,j}}R)_{i,j\in(Q_\sigma)_0}\subset \M_{2n-3}(K[t]).$$
Since \eqref{eq:length} implies that $(t^{\ell^{\theta}_{i,j}}R) \cdotp (t^{\ell^{\theta}_{j,k}}R) \subset t^{\ell^{\theta}_{i,k}}R$,
${\Gamma'_\sigma}$ is a $R$-subalgebra of $\M_{2n-3}(K[t])$. 
Moreover, we get the following proposition.

\begin{proposition}
\label{pro:tiled order}
The frozen Jacobian algebra $\Gamma_\sigma$ is isomorphic to ${\Gamma'_\sigma}$ as an $R$-algebra.
\end{proposition}
\begin{proof}
We consider the $R$-linear map 
\begin{align*}
\Psi : {\Gamma'_\sigma}  &\to \Gamma_\sigma\\
 r t^{\ell^{\theta}_{i,j}} E_{i,j} &\mapsto  \phi(r)w_0(i,j),
 \end{align*}
for any $r\in R$ and $i,j\in (Q_\sigma)_0$, where $\phi$ is the homomorphism defined in \eqref{equ:mor} and $E_{i,j}$ is the matrix which has 1 in the $(i,j)$-entry and $0$ elsewhere.

{The map $\Psi$ sends the $R$-basis $\{t^{\ell^\theta_{i,j}} E_{i,j}\,|\,i, j \in (Q_\sigma)_0\}$ of $\Gamma'_\sigma$ to the $R$-basis $\{w_0(i,j) \mid i,j \in (Q_\sigma)_0\}$ of $\Gamma_\sigma$ and it} is therefore an isomorphism of $R$-modules. Moreover, for $i,j,i',j'\in (Q_\sigma)_0$, we have 
\begin{align*}
\Psi(t^{\ell^{\theta}_{i,j}}E_{i,j} t^{\ell^{\theta}_{i',j'}} E_{i',j'}) 
&= \Psi(\delta_{j,i'} t^{\ell^{\theta}_{i,j}+\ell^{\theta}_{i',j'}} E_{i,j'})  = \delta_{j,i'} C^{{\ell^{\theta}_{i,j}/n}+\ell^{\theta}_{i',j'}/n-\ell^{\theta}_{i,j'}/n} w_0(i,j')\\
& = w_0(i,j)w_0(i',j')  = \Psi(t^{\ell^{\theta}_{i,j}}E_{i,j})  \Psi(t^{\ell^{\theta}_{i',j'}}E_{i',j'}).
 \end{align*} 
So $\Psi$ is an isomorphism of $R$-algebras.
\end{proof}
By applying elementary matrix transformations to ${\Gamma'_\sigma}$, we get the following corollary.
\begin{theorem}
\label{cor:tiled order}
Let $k\in (Q_\sigma)_0$. The frozen Jacobian algebra $\Gamma_\sigma$ is isomorphic to the tiled $R$-order (Definition \ref{def:tiled order}) $$ \left((x^{d_{k,j}^i } )\right)_{i,j \in \{1, \ldots, 2n-3\}}.$$
\end{theorem}
\begin{proof}
Consider the {diagonal }matrix 
$$X_k={\diag\left( t^{\ell^{\theta}_{k,i} } \right)_{1 \le i \le 2n-3.}}$$
Then we have 
\begin{align*}
X_k {\Gamma'_\sigma} X_k^{-1} &= \left(t^{\ell^{\theta}_{k,i} }(t^{\ell^{\theta}_{i,j}}R)t^{-\ell^{\theta}_{k,j} }\right)_{i,j \in \{1, \ldots, 2n-3\}} \\
 &= \left(t^{\ell^{\theta}_{k,i} +\ell^{\theta}_{i,j} -\ell^{\theta}_{k,j} }R\right)_{i,j \in \{1, \ldots, 2n-3\}}\\ 
 &= \left((x^{d_{k,j}^i })\right)_{i,j \in \{1, \ldots, 2n-3\}}.
 \end{align*}
 Since $\Gamma_\sigma$ is isomorphic to ${\Gamma'_\sigma}$ as an $R$-order by Proposition \ref{pro:tiled order}, it follows that $\Gamma_\sigma$ is a tiled $R$-order which is isomorphic to $X_k{\Gamma'_\sigma} X_k^{-1}$.
\end{proof}

Let $e_F$ be the sum of the {idempotents} at all frozen vertices in ${Q_\sigma}$. Consider the suborder $$\Lambda_\sigma:=e_F \Gamma_\sigma e_F.$$ We get the following theorem.
\begin{theorem}
 \label{thm:frozen part}
The two $R$-orders $\Lambda_\sigma$ and $\Lambda$ are isomorphic.
\end{theorem}

\begin{proof}
Consider the matrix 
 $$X' = \diag \left(x^{\delta_{i,n}}\right)_{1 \leq i \leq 2n-3} X_1$$
 where $X_1$ is defined in the proof of Theorem \ref{cor:tiled order}. We get
 $$X' \Gamma'_\sigma X'^{-1} = \left((x^{d_{1,j}^i + \delta_{i,n}-\delta_{j,n}})\right)_{i,j\in\{1,\ldots,2n-3\}}.$$
Under the isomorphism $\Psi : {\Gamma'_\sigma}  \to \Gamma_\sigma$, the idempotent $e_F\in \Gamma_\sigma$ corresponds to $$E_F:=\begin{bmatrix}
\Id_{n} & 0 \\
0 & 0_{n-3}
  \end{bmatrix}_{(2n-3) \times (2n-3)}\in \Gamma'_\sigma.$$
   Thus we have $$\Lambda_\sigma \cong E_F {\Gamma'_\sigma} E_F \cong  X' E_F {\Gamma'_\sigma} E_F X'^{-1} \cong E_F X' {\Gamma'_\sigma} X'^{-1}E_F\cong \left((x^{d^i_{1,j}+\delta_{i,n}-\delta_{j,n}})\right)_{i,j\in\{1,\ldots,n\}}.$$
 We only have to show that $(x^{d^i_{1,j}+\delta_{i,n}-\delta_{j,n}})$ coincides with the $(i,j)$-entry of $\Lambda$.
 
 According to our notation, if $1\le i\le n-1$, the vertex $i$ of $Q_\sigma$ corresponds to the side $(P_i, P_{i+1})$ of the polygon and the vertex $n$ corresponds to $(P_1,P_n)$. Using Proposition \ref{pro:length}, if we denote $p^+ = p+1$ if $p < n$ and $n^+ = 1$, we get
 \begin{align*}d^i_{1,j} &= \frac{\ell^\theta_{1,i} + \ell^\theta_{i,j} - \ell^\theta_{1,j}}{n} \\ &= \min(\delta_{2>i}+\delta_{1>i^+}, \delta_{1>i}+\delta_{2>i^+}) + \min(\delta_{i^+>j}+\delta_{i>j^+}, \delta_{i>j}+\delta_{i^+>j^+}) \\& \quad - \min(\delta_{2>j}+\delta_{1>j^+}, \delta_{1>j}+\delta_{2>j^+}) \\
 &= 0 + \min(\delta_{i^+>j}+\delta_{i>j^+}, \delta_{i>j}+\delta_{i^+>j^+}) - 0 \\
  &= {\begin{cases}
1-\delta_{j,n-1}-\delta_{j,n}, & \text{if } i = n, \\
1-\delta_{i,1}-\delta_{i,n}, & \text{if } j = n, \\
\delta_{i > j} + \delta_{i > j+1} , & \text{if } n \notin \{i,j\}
\end{cases}} \\
 &=\delta_{i > j} + \delta_{i > j+1} + \delta_{j,n} \delta_{i > 1} - \delta_{i, n}. 
 \end{align*}
 
  Thus, $d^i_{1,j} + \delta_{i,n} - \delta_{j,n} = \delta_{i > j} + \delta_{i > j+1} - \delta_{j,n} \delta_{i,1}$ and $\Lambda = ((x^{d_{1,j}^i  +\delta_{i,n} - \delta_{j,n}}))_{i,j \in \{1, \ldots, n\}}$ holds.
 \end{proof}

\section{Cohen-Macaulay Modules over $\Lambda$}
\label{s:Cohen-Macaulay modules over orders}
Let $\Lambda$ be the $R$-order defined in \eqref{the order}. The aim of this section is to study the representation theory of $\Lambda$ and its connection to triangulations of $P$ and the cluster category of type $A_{n-3}$.
In particular, we classify all Cohen-Macaulay $\Lambda$-modules and construct a bijection between the set of the isomorphism classes of all indecomposable Cohen-Macaulay $\Lambda$-modules and the set of all sides and diagonals of the polygon $P$. We then  show that the stable category $\underline{\CM}(\Lambda)$ of Cohen-Macaulay $\Lambda$-modules is $2$-Calabi-Yau and $\underline{\CM}(\Lambda)$ is triangle-equivalent with the cluster category of type $A_{n-3}$.

Throughout, we denote $\DK:=\Hom_K(-,K)$, $\DR:=\Hom_{{R}}(-,{R})$ and $(-)^*:=\Hom_{\Lambda}(-,\Lambda)$.

\subsection{Classification of Cohen-Macaulay $\Lambda$-modules}
\label{ss:CM modules}
It is easy to check that the quiver of $\Lambda$ is given as follows:
 
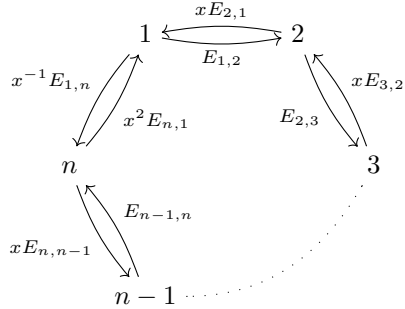
\begin{center}  
  \begin{tikzpicture}[commutative diagrams/every diagram]
  \node (a) at (0:2)      {$3$};
  \node (b) at (0+60:2)   {$2$};
  \node (c) at (0+60*2:2) {$1$};
  \node (d) at (0+60*3:2) {$n$};
  \node (e) at (0+60*4:2) {$n-1$};
  \path[commutative diagrams/.cd, every arrow, every label]
  (a) edge[out=110,in=-50,commutative diagrams/rightarrow] node[swap] {$xE_{3,2}$} (b)
  (b) edge[out=170,in=10,commutative diagrams/rightarrow] node[swap] {$xE_{2,1}$}  (c)
  (c) edge[out=230,in=70,commutative diagrams/rightarrow] node[swap] {$x^{-1}E_{1,n}$}  (d)
  (d) edge[out=-70,in=130,commutative diagrams/rightarrow] node[swap] {$xE_{n,n-1}$} (e)
  (e) edge[out=110,in=-50,commutative diagrams/rightarrow]  node[swap] {$E_{n-1,n}$}(d)
  (d) edge[out=50,in=-110,commutative diagrams/rightarrow]  node[swap] {$x^2E_{n,1}$} (c)
  (c) edge[out=-10,in=190,commutative diagrams/rightarrow] node[swap] {$E_{1,2}$}  (b)
  (b) edge[out=-70,in=130,commutative diagrams/rightarrow] node[swap] {$E_{2,3}$} (a)
  (a) edge[out=-120,in=0,loosely dotted,commutative diagrams/path] (e);
\end{tikzpicture}
\end{center}
This helps to prove the following theorem.

\begin{theorem}
\label{th:indecomposable modules}
\begin{enumerate}
  \item A Cohen-Macaulay $\Lambda$-module is indecomposable if and only if it is isomorphic to $$(i,j)=[\overbrace{R \cdots R}^{i} \overbrace{(x)  \cdots (x)}^{j-i} \overbrace{(x^2)  \cdots (x^2)}^{n-j}]^t$$ for some $ 1\le i<j\le n$.

  \item Any Cohen-Macaulay $\Lambda$-module is isomorphic to $\bigoplus_{i,j}(i,j)^{l_{i,j}}$ for some non-negative integers $l_{i,j}$. Moreover, $l_{i,j}$ are uniquely determined.
\end{enumerate}
\end{theorem}

\begin{remark}
Theorem \ref{th:indecomposable modules} shows that Krull-Schmidt-Azumaya property is valid in this case. This is interesting by itself since our base ring $R=K[x]$ is not even local, and in such a case, Krull-Schmidt-Azumaya property usually fails.
\end{remark}

\begin{figure}
{\tiny \[
\begin{tikzcd}[column sep=0ex, row sep=2ex]
\phantom{X}\arrow[dashed,-]{dd} && & & & &&&\phantom{X}\\
&(1,2) \drar & &(2,3)\drar\arrow[dotted,-]{rrrr}&&&&(\frac{n}{2},{\frac{n}{2}+1})\drar&\\
(2,n)\urar\drar\arrow[dashed,-]{dd} & &(1,3)\urar\drar&&(2,4)\arrow[dotted,-]{rr} &&(\frac{n}{2}-1,{\frac{n}{2}+1}) \urar\drar & &({\frac{n}{2},{\frac{n}{2}+2}})\arrow[dashed,-]{uu}\arrow[dashed,-]{dd}\\
&(3,n)\urar\drar&&(1,4)\urar\drar\arrow[dotted,-]{rrrr}& & & &({\frac{n}{2}-1,{\frac{n}{2}+2}})\drar\urar&\\ 
(3,n-1)\urar\drar \arrow[dashed,-]{dddddd}& &(4,n)\urar\drar\arrow[dotted,-]{dddddd} &&(1,5) \arrow[dotted,-]{rr} &&(\frac{n}{2}-2,{\frac{n}{2}+2})  \urar\drar & &({\frac{n}{2}-1,{\frac{n}{2}+3}})\arrow[dashed,-]{dddddd}\\
&(4,n-1)\arrow[dotted,-]{dddd} \urar&&(5,n)\urar\arrow[dotted,-]{dddd} \arrow[dotted,-]{rrrr}&&&&({\frac{n}{2}-2,{\frac{n}{2}+3}})\arrow[dotted,-]{dddd} \urar&\\
 & \quad \\
 & \quad \\
 & \quad \\
&({\frac{n}{2},\frac{n}{2}+3})\drar&&({\frac{n}{2}+1,\frac{n}{2}+4})\drar\arrow[dotted,-]{rrrr} &&&&(2,n-1)\drar&\\
({\frac{n}{2},{\frac{n}{2}+2}})\urar\drar & &
({\frac{n}{2}+1,{\frac{n}{2}+3}})\urar\drar&&({\frac{n}{2}+2,\frac{n}{2}+4})\arrow[dotted,-]{rr} && (1,n-1)\urar\drar & &(2,n).\\
&({\frac{n}{2}+1,{\frac{n}{2}+2}})\urar&&({\frac{n}{2}+2,\frac{n}{2}+3})\urar\arrow[dotted,-]{rrrr} && & &(1,n)\urar&\\
\phantom{X}\arrow[dashed,-]{uu}& & & & & &&&\phantom{X}\arrow[dashed,-]{uu}
\end{tikzcd}
\]}
\caption{$\CM(\Lambda)$ for even $n$}
\label{CMLe}
{\tiny \[
\begin{tikzcd}[column sep=0ex, row sep=2ex]
\phantom{X}\arrow[dashed,-]{d} && & & & &&&\phantom{X}\\
(1,2) \drar\arrow[dashed,-]{dd} & &(2,3)\drar && (3,4)\arrow[dotted,-]{rrr}&&&(\frac{n+1}{2},{\frac{n+3}{2}})\drar&\\
&(1,3)\urar\drar & &(2,4)\urar\drar\arrow[dotted,-]{rrr} &&&(\frac{n-1}{2},{\frac{n+3}{2}}) \urar\drar & &({\frac{n+1}{2},{\frac{n+5}{2}}})\arrow[dashed,-]{uu}\arrow[dashed,-]{dd}\\
(3,n)\urar\drar \arrow[dashed,-]{dddddd}&&(1,4)\urar\drar\arrow[dotted,-]{dddddd}&&(2,5)\arrow[dotted,-]{rrr}& & &({\frac{n-1}{2},{\frac{n+5}{2}}})\drar\urar\arrow[dotted,-]{dddddd}&\\ 
&(4,n)\urar\arrow[dotted,-]{dddd}& &(1,5) \urar\arrow[dotted,-]{dddd}\arrow[dotted,-]{rrr} &&&(\frac{n-3}{2},{\frac{n+5}{2}})  \urar\arrow[dotted,-]{dddd} & &({\frac{n-1}{2},{\frac{n+7}{2}}})\arrow[dashed,-]{dddd}\\
& \quad \\
 & \quad \\
 & \quad \\
&({\frac{n+1}{2},\frac{n+7}{2}})\drar&&({\frac{n+3}{2},\frac{n+9}{2}})\drar\arrow[dotted,-]{rrr} &&&(2,n-1)\drar&&(3,n) \arrow[dashed,-]{dd}\\
({\frac{n+1}{2},{\frac{n+5}{2}}})\urar\drar & &
({\frac{n+3}{2},{\frac{n+7}{2}}})\urar\drar&&({\frac{n+5}{2},\frac{n+9}{2}})\arrow[dotted,-]{rrr} && &(2,n)\drar \urar &.\\
&({\frac{n+3}{2},{\frac{n+5}{2}}})\urar&&({\frac{n+5}{2},\frac{n+7}{2}})\urar\arrow[dotted,-]{rrr} && &(1,n)\urar& &(1,2)\\
\phantom{X}\arrow[dashed,-]{uu}& & & & & &&&\phantom{X}\arrow[dashed,-]{u}
\end{tikzcd}
\]}
\caption{$\CM(\Lambda)$ for odd $n$}
\label{CMLo}
\end{figure}

\begin{remark}
\label{remark:quiver}
(1) When $n$ is an even number, the category $\CM(\Lambda)$ is presented by the quiver of Figure \ref{CMLe}. This quiver has the shape of a M{\"o}bius strip. Its first and last rows consist of $n/2$ projective-injective Cohen-Macaulay $\Lambda$-modules respectively, with $\lceil(n-3)/2\rceil$ $\tau$-orbits between them.

(2) When $n$ is an odd number,  the category $\CM(\Lambda)$ is presented by the quiver of Figure \ref{CMLo}. After completion, this is the Auslander-Reiten quiver of $\CM(\Lambda \tens_R K\llbracket {x} \rrbracket)$, \emph{c.f.} proof of Theorem \ref{th:graded modules}.
\end{remark}

In order to prove Theorem \ref{th:indecomposable modules}, we need the following lemmas.
\begin{lemma} \label{trigo}
For any matrix $M\in \M_k(R)$, there exists an invertible matrix $G\in \GL_k(R)$ such that $MG$ is upper triangular.
\end{lemma}

\begin{proof}
Take two elements $A, B \in R$ such that $B \neq 0$. Write $A = BQ + S$ the usual Euclidean division of $A$ by $B$. We have
  $$\vech{A}{B} \mat{1}{0}{-Q}{1} = \vech{S}{B},$$
  so by using the usual Euclidean algorithm, there is an invertible $2 \times 2$ matrix $G_0$ such that
  $$\vech{A}{B} G_0 = \vech{0}{\gcd(A, B)}.$$
  This observation permits to apply an analogous to the Gaussian elimination method to the matrix $M$ by right multiplication to obtain an upper triangular matrix.
\end{proof}

\begin{definition}
 An element $v \in R^k$ is called \emph{basis-adapted} if it is part of an $R$-basis of $R^k$. It is equivalent to say that the standard coordinates of $v$ generate $R$ as an ideal.

 Let $M \in \M_k(R)$. An element $v \in R^k$ is called \emph{critical of degree $d$ for $M$} if $v$ is basis-adapted and $M v$ is of the form $x^d w$ where $w$ is basis-adapted.
\end{definition}

\begin{lemma} \label{critic}
If $M$ and $N$ are two upper triangular matrices in $\M_k(R)$ satisfying that $MN=x^2 \Id$, then there exists two invertible upper triangular matrices $G$ and $H$ such that $GMH$ is a diagonal matrix with diagonal entries in $\{1, x\}$ or $M$ admits a critical vector of degree $2$.
\end{lemma}

\begin{proof}
Up to scalar multiplications, we can suppose that the diagonal entries of $M$ and $N$ are powers of $x$.

 If one of the diagonal entries of $M$ is $x^2$, then the corresponding entry of $N$ is $1$. Thus, $N$ admits a critical vector $v$ of degree $0$. As $MN = x^2 \Id$, $Nv$ is a critical vector of degree $2$ of $M$. 

 Suppose now that all the diagonal entries of $M$ are $x$. Up to right multiplication by an invertible upper triangular matrix, we can suppose that all elements outside of the diagonal are scalars. It is then the same for $N$. Therefore, if $N \neq x \Id$, $N$ admits a critical vector $v$ of degree $0$ (take the basis vector $e_i$ for a column of $N$ containing a nonzero scalar). Thus $M$ admits the critical vector $Nv$ of degree $2$.

 Suppose finally that the diagonal entries of $M$ are $x$ and $1$. Up to right and left multiplications by invertible upper triangular matrices, we can suppose that $M$ contains only scalars outside of the diagonal, and moreover that nonzero scalars $M_{i j}$ appear only if $M_{ii} = M_{j j} = x$. Thus, we can apply the previous case to the submatrix consisting of rows and columns whose diagonal entries are $x$.
\end{proof}

\begin{lemma} \label{diagon}
Consider $M_1,M_2,\ldots,M_l,N_1,N_2, \ldots, N_l \in \M_k(R)$ satisfying $M_1 M_2 \ldots M_l = x^2\Id$ and for each $i\in \{1,2,\ldots, l\}$, $M_i N_i=x\Id$. There exist $G_0=G_l, G_1,G_2, \ldots, G_{l-1}\in \GL_k(R)$ such that for each $i\in \{1,2,\ldots, l\}$, $G_i^{-1}M_iG_{i-1}$ and $G_{i-1}^{-1}N_iG_i$ are diagonal matrices with diagonal entries in $\{1,x\}$.
\end{lemma}

\begin{proof}
As each $M_i$ is invertible if we pass to the fraction field of $R$, it is easy to see that the relation $$M_i M_{i+1} \cdots M_l M_1 \cdots M_{i-1} = x^2 \Id$$ is satisfied for every $i\in \{1,2,\ldots, l\}$. 

 Thus we can suppose that there is an $m \in \NN^*$ such that $M_1 M_2 \cdots M_m$ has a critical vector $v$ of degree $2$ while $M_1 M_2 \cdots M_{m-1}$ and $M_2 M_3 \cdots M_m$ have no critical vector of degree $2$. 

 Up to replacing $M_1$ by some $G M_1$, $M_l$ by $M_l G^{-1}$, $M_m$ by $M_m H$ and $M_{m+1}$ by $H^{-1} M_{m+1}$ (and the $N$'s similarly) for $G, H \in \GL_k(R)$, we can then suppose that 
 $$M_1 M_2 \cdots M_m = \mat{x^2}{L}{0}{M'}$$
 for some matrices $M'$ and $L$ (and $v$ is the first basis vector). Moreover, applying Lemma \ref{trigo}, we can suppose that $M'$ is upper triangular. 

 Using again Lemma \ref{trigo}, for some $G \in \GL_k(R)$, replacing $M_1$ by $M_1 G$, $N_1$ by $G^{-1} N_1$, $M_2$ by $G^{-1} M_2$ and $N_2$ by $N_2 G$, we can suppose that $M_1$ and $N_1$ are upper triangular. Then, in the same way, without modifying $M_1$ and $N_1$, we can suppose that $M_2$ and $N_2$ are upper triangular and so on until $M_{m - 1}$. Thus, as $M_1 \cdots M_m$ is upper triangular, $M_m$ and $N_m$ are also upper triangular. Note that we did not change $M_1 M_2 \cdots M_m$. We can do the same for $M_{m+1}$, \dots , $M_{l-1}$ without changing the product $M_{m+1} \cdots M_l$ and get for free that $M_l$ is upper triangular.

 As a consequence, for $i \in \llbracket 1, l \rrbracket$ we can write
 $$M_i = \mat{t_i}{L_i}{0}{M'_i}$$
 where the $t_i$'s are $1$ or $x$. More precisely, $t_1 = t_m = x$ and $t_j = 1$ for $j \notin \{1,m\}$. Indeed, the first basis vector is critical of degree $2$ for $M_1 \cdots M_m$ but $M_1 \cdots M_{m-1}$ and $M_2 \cdots M_m$ have no critical vectors of degree $2$. Thanks to Lemma \ref{critic}, there are two upper triangular invertible matrices $G$ and $H$ such that $G M_2 \dots M_m H$ is diagonal with $1$ and $x$ on the diagonal. So we suppose that $M_2 \dots M_m$ is diagonal with $1$ and $x$ on the diagonal.
 
 As $M_2 \cdots M_m$ and $M_2 \cdots M_l M_1$ are diagonal, $M_{m+1} \cdots M_l M_1$ is also diagonal. 

 As upper triangular matrices $M_2$, $M_3$, \dots, $M_{m-1}$, $M_{m+1}$, \dots, $M_l$ have $1$ in the upper left corner, we can, up to multiplication on the right by invertible upper triangular matrices suppose that they are of the form
 $$M_i = \mat{1}{0}{0}{M'_i}$$
 without changing the products $M_2 \cdots M_{m}$ and $M_{m+1} \cdots M_\ell M_1$

 As $M_2 M_3 \cdots M_m$ is diagonal, we automatically get that $M_m$ is of the form
  $$M_m = \mat{x}{0}{0}{M'_m}$$
 and in the same way 
 $$M_1 = \mat{x}{0}{0}{M'_1}.$$
 Applying inductively the same method on the $M'_i$'s, we prove the existence of the expected $G_i$'s.
\end{proof}

\begin{proof}[Proof of Theorem \ref{th:indecomposable modules}]
 The quiver of $\Lambda$ is the double quiver of a cycle of length $n$. Call the arrows in one direction $\alpha_i$ and the other $\beta_i$. A quick analysis shows that $\alpha_i \beta_i = x$, $\beta_i \alpha_i = x$ and $\alpha_n \alpha_{n-1} \dots \alpha_1 = x^2$ so using Lemma \ref{diagon} we get that any Cohen-Macaulay $\Lambda$-module is ``diagonalizable''. In particular, any indecomposable has rank $1$ over $R$ at each vertex. Looking carefully at the proof of Lemma \ref{diagon}, we get for free that there are $1$ or $2$ changes of degree when we multiply by $\alpha$ (they are determined by the positions of $1$ and $m$ in the previous proof, up to the cyclic rotation that has been done at the beginning). 

 For the uniqueness of the decomposition, just notice that the endomorphism algebra of each indecomposable object is $R$. Moreover, any endomorphism factorizing through another indecomposable is in the ideal $(x)$. Thus, if we denote $\hat \Lambda = K\llbracket x \rrbracket \tens_R \Lambda$, and consider the functor $ K\llbracket x \rrbracket \tens_R -: \mod \Lambda \rightarrow \mod \hat \Lambda$, non-isomorphic indecomposable objects are mapped to non-isomorphic objects, which are also indecomposable using explicitly the classification we did. Moreover, the endomorphism rings of the objects $K \llbracket x \rrbracket \tens_R (i, j)$ are local so we get the uniqueness of the decomposition of objects in the essential image of the functor $K\llbracket x \rrbracket \tens_R -$. Therefore, we proved the uniqueness of the decomposition in $\CM(\Lambda)$.
\end{proof}

In the following, let $\mathcal{S}$ be the set of all sides and diagonals of the polygon $P$ and $\mathcal{E}$ be the set of the isomorphism classes of all indecomposable Cohen-Macaulay $\Lambda$-modules. 
According to Theorem \ref{th:indecomposable modules}, each indecomposable Cohen-Macaulay $\Lambda$-module is of the form $(s,t)$ for $1\le s < t \le n$. On the other hand, each side or diagonal of the polygon $P$ can also be denoted as a pair $(P_s,P_t)$ for $(1\le s<t \le n)$.
In particular, they correspond bijectively.
This bijection $\epsilon: \mathcal{S} \rightarrow \mathcal{E}$ can be given in terms of the frozen Jacobian algebras.
 
\begin{theorem}
\label{thm:edges and modules}
For any triangulation $\sigma$ and $(P_s,P_t) \in \sigma$ ($1\le s<t \le n$), the vertex $j=(P_s,P_t)$ satisfies that $$e_F \Gamma_\sigma e_j \cong (s,t)  = [\overbrace{R \cdots R}^{s} \overbrace{(x)  \cdots (x)}^{t-s} \overbrace{(x^2)  \cdots (x^2)}^{n-t}]^{\operatorname{t}}$$ as $\Lambda$-modules.
Hence, $\restr{\epsilon}{\sigma}$ can be rewritten as
   \begin{eqnarray*}
  \restr{\epsilon}{\sigma}: \sigma &\to &  \mathcal{E} \\
     j & \mapsto & e_F {\Gamma}_\sigma e_j.
   \end{eqnarray*}
\end{theorem}

\begin{proof}
Let $\sigma$ be a triangulation with $(P_s,P_t) \in \sigma$ and the vertex $j\in (Q_\sigma)_0$ corresponding to $(P_s,P_t)$.
By taking $X'$ as in Theorem \ref{thm:frozen part}, we have 
\begin{align*}
& e_F \Gamma_\sigma e_j  \cong  E_F X'{\Gamma'_\sigma} X'^{-1} E_{j,j}  \\
 \cong  &  \begin{bmatrix}
     (x^{d_{1,j}^1  +\delta_{1,n} - \delta_{j,n}}) & (x^{d_{1,j}^2  +\delta_{2,n} - \delta_{j,n}}) &\cdots &(x^{d_{1,j}^n  +\delta_{n,n} - \delta_{j,n}})
   \end{bmatrix}^{\operatorname{t}}.
\end{align*} 
 
 In the rest, we calculate $d_{1,j}^i  +\delta_{i,n} - \delta_{j,n}$ for $1\le i \le n$.
If $j\in F$, the computation has been done in the proof of Theorem \ref{thm:frozen part}.

Assume that $j \in (Q_\sigma)_0 \smallsetminus F$.
First, suppose that $1 \le i \le n-1$. By \eqref{eq:length}, we have $d^i_{1,j}=\ell^{\theta}_{1,i}/n+\ell^{\theta}_{i,j}/n-\ell^{\theta}_{1,j}/n$.
Recall that $1=(P_1,P_2)$, $i=(P_i,P_{i+1})$ and $j=(P_s,P_t)$. By Proposition \ref{pro:length}, we get 
   \begin{align*}
   \ell^{\theta}_{1,i}&= i+(i+1)-1-2+n\min\{\delta_{1>i}+\delta_{2>i+1},\delta_{1>i+1}+\delta_{2>i}\}=2i-2,\\
   \ell^{\theta}_{i,j}&= s+t-i-(i+1)+n\min\{\delta_{i>s}+\delta_{i+1>t},\delta_{i>t}+\delta_{i+1>s}\}, \\
   \ell^{\theta}_{1,j}&= s+t-1-2+n\min\{\delta_{1>s}+\delta_{2>t},\delta_{1>t}+\delta_{2>s}\}=s+t-3. 
   \end{align*} 
Since $t \ge s+2$, we have 
 \begin{align*}
 {d_{1,j}^i  +\delta_{i,n} - \delta_{j,n}} = d_{1,j}^i  &=  \min\{\delta_{i>s}+\delta_{i+1>t},\delta_{i>t}+\delta_{i+1>s}\} \\
  &=\begin{cases}
     0, &\text{if } 1\le i \le s, \\
     1, &\text{if } s < i \le \min\{t,n-1\}, \\
     2, &\text{if } t < i \le n-1.
   \end{cases}
 \end{align*} 
 Secondly, when $i = n$, one can calculate 
 \begin{align*}
 {d_{1,j}^{i}  +\delta_{i,n} - \delta_{j,n}} =d_{1,j}^i +1
 &= \min\{\delta_{1>s}+\delta_{n>t},\delta_{1>t}+\delta_{n>s}\} +1\\
 &= \begin{cases}
   2, &\text{if } t<n,\\
   1, &\text{if } t=n.
   \end{cases} 
 \end{align*}
Therefore, $e_F \Gamma_\sigma e_j  \cong  [\overbrace{R \cdots R}^{s} \overbrace{(x)  \cdots (x)}^{t-s} \overbrace{(x^2)  \cdots (x^2)}^{n-t}]^{\operatorname{t}}$ holds.
\end{proof}

\begin{theorem}
\label{th:endomorphism}
For a triangulation $\sigma$, consider the Cohen-Macaulay $\Lambda$-module $T_\sigma$ defined by $$T_\sigma:=e_F\Gamma_\sigma\cong \bigoplus_{(P_s,P_t)\in \sigma} \epsilon (P_s,P_t)$$ where the sum runs over the sides and diagonals appearing in $\sigma$.
    Through right multiplication, there is a canonical isomorphism $$\End_{\Lambda}(T_\sigma) \cong \Gamma_\sigma^{\operatorname{op}}.$$
  \end{theorem}
\begin{proof}
As $\Lambda$-modules, $$T_\sigma \cong \Lambda \oplus e_F\Gamma_\sigma (1-e_F).$$
Then we have 
\begin{align*}
\End_{\Lambda}(T_\sigma)
&\cong \begin{bmatrix}
\Hom_{\Lambda}(\Lambda,e_F \Gamma_\sigma) & \Hom_{\Lambda}(e_F\Gamma_\sigma (1-e_F),e_F \Gamma_\sigma)
\end{bmatrix} \\ &\cong \begin{bmatrix}
e_F\Gamma_\sigma^{\operatorname{op}} & \Hom_{\Lambda}(e_F\Gamma_\sigma (1-e_F),e_F \Gamma_\sigma)
\end{bmatrix}_.
\end{align*}

By Proposition \ref{pro:tiled order}, as an $R$-order, $\Gamma_\sigma$ is isomorphic to ${\Gamma'_\sigma} = (t^{\ell^{\theta}_{i,j}}R)_{i,j\in(Q_\sigma)_0}$. So $$\Hom_{\Lambda}(e_F\Gamma_\sigma e_h,e_F\Gamma_\sigma e_k) \cong \Hom_{\Lambda}(E_F {\Gamma'_\sigma} E_{h,h},E_F {\Gamma'_\sigma} E_{k,k})$$ holds for $h,k\in (Q_\sigma)_0$.

By right multiplication, there is a canonical embedding \[E_{h,h} \Gamma'^{\operatorname{op}}_\sigma E_{k,k} \hookrightarrow \Hom_{\Lambda}(E_F {\Gamma'_\sigma} E_{h,h},E_F {\Gamma'_\sigma} E_{k,k}).\] 
Since $$\ell^{\theta}_{h,k} = \max_{i\in F} \{ \ell^{\theta}_{i,k} -\ell^{\theta}_{i,h} \}$$ holds by Theorem \ref{th:induction}(2), this embedding is in fact surjective. 
Thus we have 
\begin{equation}
\label{eq:nonfrozen endomorphism}
\Hom_{\Lambda}(e_F\Gamma_\sigma e_h,e_F\Gamma_\sigma e_k) \cong e_h \Gamma_\sigma^{\operatorname{op}} e_k.
\end{equation}
Therefore, the right multiplication induces an isomorphism $\End_{\Lambda}({T_\sigma}) \cong \Gamma_\sigma^{\operatorname{op}}.$
\end{proof}

Next, we study the non-split short exact sequences of indecomposable Cohen-Macaulay $\Lambda$-modules and the relation with diagonals of triangulations.
 \begin{lemma}
 \label{pro:extension}
The non-split extensions of indecomposable Cohen-Macaulay $\Lambda$-modules are of the form 
$$0\lra (s,t) \lra (s,t')\oplus(s',t) \lra (s',t')\lra 0$$
with $s<s'<t<t'$,
or $$0\lra (s,t) \lra (s',s)\oplus(t',t) \lra (s',t')\lra 0$$ 
with $s'<s<t'<t$.
 \end{lemma}
 \begin{proof}
 By using the classification of indecomposable Cohen-Macaulay $\Lambda$-modules in Theorem \ref{thm:edges and modules}, since the rank over $R$ of any indecomposable object of $\CM(\Lambda)$ is $n$, any extension of indecomposable Cohen-Macaulay $\Lambda$-modules can be written as $$0\lra (u,v) \lra (u',v')\oplus(u'',v'') \lra (u''',v''')\lra 0.$$
  Multiplying powers of $x$, it can be written in the following form:
$$0\lra  \vvect{(x^{a_1})\\(x^{a_2})\\ \vdots\\(x^{a_n})} \xlongrightarrow{\vvect{P_1 \\ P_2}} \vvect{(x^{b_1})\\(x^{b_2})\\ \vdots\\(x^{b_n})}\oplus \vvect{(x^{c_1})\\(x^{c_2})\\ \vdots\\(x^{c_n})}\xlongrightarrow{\hvect{Q_1 & - Q_2}}\vvect{(x^{d_1})\\(x^{d_2})\\ \vdots\\(x^{d_n})} \lra 0,$$
where $P_1$, $P_2$, $Q_1$ and $Q_2$ are polynomials with a nonzero constant term such that $P_1 Q_1 = P_2 Q_2$ and $a_{i+1}-a_i,b_{i+1}-b_i,c_{i+1}-c_i,d_{i+1}-d_i \in \{0,1\}$ for all $1\le i \le n$.

As the morphisms are well-defined, we have $a_i\ge b_i$, $a_i\ge c_i$, $b_i\ge d_i$ and $c_i\ge d_i$.
Since $\hvect{Q_1 & - Q_2}$ is surjective, it follows that $d_i=b_i$ or $d_i=c_i$ for any $i=1,\ldots,n$. Thus $d_i=\min\{b_i, c_i\}$. Moreover $\gcd(Q_1, Q_2) = 1$.

It is clear that $${\ker \hvect{Q_1 & -Q_2} = \vvect{(x^{\max\{b_1,c_1\})} \\(x^{\max\{b_2,c_2\}}) \\ \vdots\\ (x^{\max\{b_n,c_n\}})} \hvect{Q_2 & Q_1}}.$$
Since the sequence is exact, it follows that there is a polynomial $W_{i}$ satisfying $x^{a_i}W_{i}P_1=x^{\max\{{b_i,c_i}\}}Q_2$ and $x^{a_i}W_i P_2=x^{\max\{b_i,c_i\}}Q_1$ for $1\le i \le n$. As the smallest degree of terms of $W_i$ is non-negative, we have $a_i\le \max\{b_i,c_i\}$. Thus, we get $a_i=\max\{b_i, c_i\}$. As $\gcd(Q_1, Q_2) = 1$, we also get, up to scalar renormalization, that $W_i = 1$, $P_1 = Q_2$ and $P_2 = Q_1$.

When the middle term of an exact sequence is of the form $(i,j) \oplus (i',j')$ with $i<j$ and $i'<j'$, according to the above observation, the exact sequence can only be $$ 0\lra (\min\{i,i'\}, \min\{j,j'\}) \lra (i,j) \oplus (i',j') \lra (\max\{i,i'\}, \max\{j,j'\}) \lra 0.$$

When the middle term of an exact sequence is of the form $(i,j) \oplus x(i',j')$ with $i<j$ and $i'<j'$, according to the above observation, the exact sequence can only be $$ 0\lra x(\min\{i',j\}, j') \lra (i,j) \oplus x(i',j') \lra (i, \max\{i',j\}) \lra 0.$$

When the middle term of an exact sequence is of the form $(i,j) \oplus x^d(i',j')$ with $i<j, i'<j'$ and $d\ge 2$, according to the above observation, the exact sequence can only be $$ 0\lra x^d(i', j') \lra (i,j) \oplus x^d(i',j') \lra (i, j) \lra 0.$$

Observe finally, thanks to the previous analysis, that if the middle term of the short exact sequence was isomorphic to the direct sum of the two external terms, the short exact sequence would split.
 \end{proof}
According to Lemma \ref{pro:extension}, we can get the following proposition.
 \begin{proposition}
 \label{cor:extensioncross}
 For any two indecomposable Cohen-Macaulay $\Lambda$-modules $(s,t)$ and $(s',t')$, the following statements are equivalent:
 \begin{enumerate}
 \item $\Ext^1_{\Lambda}((s,t),(s',t'))\neq 0$;
  \item $\Ext^1_{\Lambda}((s',t'),(s,t))\neq 0$;
  \item the two diagonals $(P_s,P_t)$ and $(P_{s'},P_{t'})$ are crossing.
 \end{enumerate} 
 \end{proposition}
 \begin{proof}
 By Lemma \ref{pro:extension}, (2) is equivalent to $s<s'<t<t'$ or $s'<s<t'<t$ holds. This is clearly equivalent to (3).
 \end{proof}

We now need to recall the definition of cluster tilting objects.
\begin{definition}
Let $\cc$ be a triangulated or exact category. An object $T$ in $\cc$ is said to be \emph{cluster tilting} if $$\add T=\{Z\in \cc \mid \Ext^1_\cc(T,Z)=0 \}=\{Z\in \cc \mid \Ext^1_\cc(Z,T)=0 \},$$ where $\add T$ is the set of finite direct sums of direct summands of $T$.
\end{definition}

 \begin{theorem}
  \label{thm:cluster tilting}
 The map $\sigma \to T_\sigma$ gives a one-to-one correspondence between the set of triangulations of the polygon and the set of isomorphism classes of basic cluster tilting objects in $\CM(\Lambda)$.
\end{theorem}
\begin{proof}
Let $\sigma$ be a set of diagonals and $T_\sigma$ the corresponding object in $\CM(\Lambda)$.
By Proposition \ref{cor:extensioncross}, any two diagonals in $\sigma$ are non-crossing if and only if $\Ext^1_{\Lambda}(T_\sigma,T_\sigma)= 0$.
Thus, $\sigma$ is a triangulation if and only if any diagonal which is non-crossing to all diagonals of $\sigma$ belongs to $\sigma$ if and only if any indecomposable $X \in \CM(\Lambda)$ satisfying $\Ext^1_{\Lambda}(T_\sigma,X)= 0$ belongs to $\add T_\sigma$ if and only if $T_\sigma$ is cluster tilting.
\end{proof}  
 
Theorem \ref{thm:edges and modules} and \ref{thm:cluster tilting} show that the category $\CM(\Lambda)$ is very similar to the cluster category of type $A_{n-3}$. In the rest to this section, we give an explicit connection. First, we recall some usual facts about cluster categories. The cluster category is defined in \cite{BMRRT} as follows.
 \begin{definition}
 For an acyclic quiver $Q$, the \emph{cluster category} $\cc(KQ)$ is the orbit category $\cd^{\operatorname{b}}(KQ)/F$ of the bounded derived category $\cd^{\operatorname{b}}(KQ)$ by the functor $F = \tau^{-1}[1]$, where $\tau$ denotes the Auslander-Reiten translation and $[1]$ denotes the shift functor.
 The objects in $\cc(KQ)$ are the same as in $\cd^{\operatorname{b}}(KQ)$, and the morphisms are given by
 $$\Hom_{\cc(KQ)}(X,Y)={\bigoplus_{i\in \ZZ}}\Hom_{\cd^{\operatorname{b}}(KQ)}(F^iX,Y),$$
 where $X$ and $Y$ are objects in $\cd^{\operatorname{b}}(KQ)$.
 For $f\in \Hom_{\cc(KQ)}(X,Y)$ and $g\in \Hom_{\cc(KQ)}(Y,Z)$, the composition is defined by 
 $$(g\circ f)_i= \sum_{i_1+i_2=i } g_{i_1} \circ F^{i_1}(f_{i_2})$$
 for all $i \in \ZZ$.
 \end{definition}
In \cite{Happel2}, Happel proves that $\cd^{\operatorname{b}}(KQ)$ has Auslander-Reiten triangles.
For a Dynkin quiver $Q$, he shows in \cite{Happel} that the Auslander-Reiten quiver of $\cd^{\operatorname{b}}(KQ)$ is $\ZZ\Delta$ where $\Delta$ is the underlying Dynkin diagram of $Q$.
Then the Auslander-Reiten quiver of $\cc(KQ)$ is $\ZZ\Delta/\varphi$, where $\varphi$ is the graph automorphism induced by $\tau^{-1}[1]$.
In type $A_{n-3}$, the Auslander-Reiten quiver of $\cc$ has the shape of a M{\"o}bius strip with $\lceil{(n-3)/2}\rceil$ $\tau$-orbits. As a quiver, it is the same as the quiver of $\underline{\CM}(\Lambda)$ (see Theorem \ref{th:indecomposable modules}).

Recall that a triangulated category is said to be algebraic if it is the stable category of a Frobenius category. Let us state the following result by Keller and Reiten.
\begin{theorem}[{\cite[Introduction and Appendix]{KR}}]
\label{KR}
If $K$ is a perfect field and $\cc$ an algebraic $2$-Calabi-Yau triangulated category containing a cluster tilting object $T$ with $\End_\cc(T)\cong KQ$ hereditary, then there is a triangle-equivalence $\cc(K Q)  \rightarrow \cc$.
\end{theorem}

   By using the above statements, we can show the following triangle-equivalences between cluster categories of type $A$ and stable categories of Cohen-Macaulay modules.
\begin{theorem}
\label{th:cyeq}
    Let $\Lambda$ be the $R$-order given in \eqref{the order}.
\begin{enumerate}
\item
\label{2-CY}
The stable category $\underline{\CM}(\Lambda)$ is $2$-Calabi-Yau.
\item
\label{th:equivalence}
If $K$ is perfect, then there is a triangle-equivalence $\cc(KQ) \cong \underline{\CM}(\Lambda)$ for a quiver $Q$ of type $A_{n-3}$. 
 \end{enumerate}   
\end{theorem}
\begin{proof}
We will prove (1) in the next subsection independently.

Take the triangulation $\sigma$ whose set of diagonals is $\{(P_1,P_3),(P_1,P_4),\cdots,(P_1,P_{n-1})\}.$
The full subquiver $Q$ of $Q_\sigma$ with the set of vertices $Q_{\sigma,0} \smallsetminus F$ is a quiver of type $A_{n-3}$.
Thus, we have  $$\Gamma_\sigma^{\operatorname{op}}/ (e_F) \cong (KQ)^{\operatorname{op}}.$$

By Theorem \ref{th:endomorphism}, for the cluster tilting object $T_\sigma$, we have the following isomorphism 
\begin{align*}
\underline{\End}_{\Lambda}(T_\sigma) \cong \Gamma_\sigma^{\operatorname{op}}/ (e_F) .
\end{align*} 

Then, by Theorem \ref{KR}, we have $\cc\left((KQ)^{\operatorname{op}}\right) \cong \underline{\CM}(\Lambda)$.
\end{proof}
 
 \subsection{Proof of Theorem \ref{th:cyeq}~(1)}
 \label{ss:Proof of 2-CY}
In this subsection, we prove the $2$-Calabi-Yau property of the stable category $\underline{\CM}(\Lambda)$.
Before, we recall some general definitions and facts about Cohen-Macaulay modules as follows. 
Let $R=K[x]$ be a polynomial ring over a field $K$ and $A$ an $R$-order.

\begin{definition}
\label{def:gorenstein}
 We call $X$ an \emph{injective} {Cohen-Macaulay} $A$-module if $\Ext^1_A(Y,X)=0$ for any $Y \in \CM(A)$, or equivalently, $X \in \add(\Hom_R(A^{\text{op}}, R))$.
 Denote by $\inj A$ the category of {\emph{injective}} Cohen-Macaulay $A$-modules.
 
 An $R$-order $A$ is \emph{Gorenstein} if $\Hom_{{R}}(A_A, R)$ is a projective as a left $A$-module, or equivalently, if $\Hom_R({}_A A, R)$ is projective as a right $A$-module. 
 \end{definition}

 We have an exact duality $\DR : \CM(A^{\operatorname{op}}) \to \CM(A)$. The Nakayama functor is defined here by $$\nu: \proj A \xlongrightarrow{(-)^*} \proj A^{\operatorname{op}} \xlongrightarrow{\DR} \inj A,$$ which is isomorphic to $(\DR A) \tens_{A}-$.
For any Cohen-Macaulay $A$-module $X$, consider a projective presentation $$P_1\xrightarrow{f} P_0 \to X \to 0.$$
 We apply $(-)^* : \mod A \to \mod A^{\operatorname{op}}$ to the projective presentation to get the following exact sequence: $$0\lra X^*\lra P_0^*\xlongrightarrow{f^*} P_1^*\lra \coker(f^*) \lra0.$$
 We denote $\coker(f^*)$ by $\tr X$ and we get $\im(f^*)=  \Omega \tr X$, where $\Omega$ is the syzygy functor: $\underline{\mod} A^{\operatorname{op}} \to \underline{\mod} A^{\operatorname{op}}$.
Then we apply $\DR: \CM(A^{\operatorname{op}}) \to \CM(A)$ to $$0\lra X^*\lra P_0^* \xlongrightarrow{f^*} \Omega \tr X \lra 0$$ and denote $\tau X:=\DR \Omega \tr X$. Thus, we get the exact sequence
\begin{equation}
\label{eq:tau}
  0\lra \tau X \lra \nu P_0 \lra \nu X \lra0.
  \end{equation}
 
  For an $R$-order $A$, if $K(x) \tens_R A$ is a semisimple $K(x)$-algebra, then we call $A$ an isolated singularity.
  By using the notions above, we have the following well-known results in Auslander-Reiten theory.
\begin{theorem}[\cite{FMO,almostsplitseqorder,MR0412223}]
 \label{isolated singularity th}
Let $A$ be an $R$-order. If $A$ is an isolated singularity, then 
  \begin{enumerate}
     \item  \cite[Chapter \uppercase\expandafter{\romannumeral1}, Proposition 8.3]{FMO} The construction $\tau$ gives an equivalence $\underline{\CM}(A) \to \overline{\CM}(A)$, where $\overline{\CM}(A)$ is the quotient of $\CM(A)$ by the ideal of maps which factor through injective objects.
     \item   \cite[Chapter \uppercase\expandafter{\romannumeral1}, Proposition 8.7]{FMO} For $X,Y\in \underline{\CM}(A)$, there is a functorial isomorphism $$\underline{\Hom}_{A}(X,Y) \cong \DK\Ext^1_{A}(Y, \tau X).$$
  \end{enumerate}
\end{theorem}

For Gorenstein orders, we have the following nice properties.
\begin{proposition}
 \label{gorenstein prop}
 Assume that $A$ is a Gorenstein isolated singularity, then we have
  \begin{enumerate}
    \item $\CM(A)$ is a Frobenius category;
    \item $\underline{\CM}(A)$ is a $K$-linear Hom-finite triangulated category;
    \item $\tau=\Omega \nu=[-1]\circ \nu$.
  \end{enumerate}
\end{proposition}
\begin{proof}
(1) The projective objects in $\CM(A)$ are just projective $A$-modules. They are also injective objects. Since each finitely generated $A$-module is a quotient of a projective $\Lambda$-module, it follows that $\CM(A)$ is a Frobenius category; (2) is due to \cite{Happel} and \cite[Lemma 3.3]{CM}; (3) is a direct consequence of  \eqref{eq:tau}.
\end{proof}

The order $\Lambda$ is Gorenstein. Indeed, as (graded) left $\Lambda$-modules,
$$\DR(\Lambda_\Lambda) = \Hom_R\left(
   \begin{bmatrix}
    R & R  & R & \cdots & R & R & (x^{-1}) \\
    (x) & R & R & \cdots & R & R & R \\
    (x^2) & (x) & R & \cdots& R & R & R \\
    \vdots&\vdots & \vdots &\ddots&\vdots&\vdots&\vdots\\
    (x^2) & (x^2) & (x^2) & \cdots & R  & R & R \\
    (x^2) & (x^2) & (x^2) & \cdots & (x) & R & R \\
    (x^2) & (x^2) & (x^2) & \cdots & (x^2) & (x) & R
   \end{bmatrix}
   , R \right)$$
can be identified with
$$\begin{bmatrix}
    R & (x^{-1})  & (x^{-2}) & \cdots  &(x^{-2}) & (x^{-2})& (x^{-2}) \\
    R & R & (x^{-1}) & \cdots& (x^{-2}) &(x^{-2})& (x^{-2})\\
    R & R & R & \cdots& (x^{-2}) &(x^{-2})& (x^{-2})\\
    \vdots&\vdots & \vdots &\ddots&\vdots&\vdots&\vdots\\
    R & R & R & \cdots& R  &(x^{-1})& (x^{-2}) \\
    R&R & R &\cdots& R& R&  (x^{-1})\\
    (x) & R &R & \cdots& R & R& R
   \end{bmatrix} = \Lambda G^{-1},$$
 where
  $$G=
   \begin{bmatrix}
  0           & 0                   & \ldots     & 0              & 1              & 0   \\
  0           & 0                      & \ldots     & 0              & 0              & 1   \\
   x^2      & 0                    & \ldots     & 0              & 0              & 0    \\
   0          & x^2                 & \ldots     & 0              & 0              & 0   \\
  \vdots   & \vdots       & \ddots    & \vdots      & \vdots     & \vdots \\
  0           & 0                    & \ldots      & x^2          & 0              & 0
  \end{bmatrix}.$$
  Therefore $\DR(\Lambda_\Lambda)$ is a projective (left) $\Lambda$-module.

According to Theorem \ref{isolated singularity th} and Proposition \ref{gorenstein prop},
we have $$\underline{\Hom}_{\Lambda}(X,Y) \cong \DK\underline{\Hom}_{\Lambda}(Y, \nu X)$$ for $X,Y\in \CM(\Lambda).$
Thus $\nu= (\DR\Lambda)\tens_{\Lambda}-$ is a Serre functor.
 We want to prove that $$(\DR\Lambda)\tens_{\Lambda}- \cong  \Omega^{-2}(-).$$
Thanks to the previous discussion, there is an isomorphism of $\Lambda$-modules:
  \[ f: \Lambda \to \DR(\Lambda_\Lambda) \text{, }  \mu \mapsto \mu G^{-1}. \]

  We define the automorphism $\alpha$ of $\Lambda$ by $\alpha (\lambda)= G^{-1} \lambda G$ for $\lambda \in \Lambda$. 
       The automorphism $\alpha$ corresponds to a $4\pi/n$ counterclockwise rotation of the quiver of $\Lambda$ {shown} in Remark \ref{remark:quiver}.
    In fact, if $$\lambda=
  \begin{bmatrix}
  \lambda_{11}             & \lambda_{12}           & \ldots     & \lambda_{1n-2}        &\lambda_{1n-1}      & x^{-1}\lambda_{1n}  \\
  x\lambda_{21}           & \lambda_{22}           & \ldots     & \lambda_{2n-2}        &\lambda_{2n-1}      & \lambda_{2n}   \\
  x^2\lambda_{31}       & x\lambda_{32}          & \ldots     & \lambda_{3n-2}        &\lambda_{3n-1}      & \lambda_{3n}   \\
  \vdots                        & \vdots                       & {\ddots}     & \vdots                      & \vdots                    & \vdots     \\
  x^2 \lambda_{n-21}   & x^2\lambda_{12}      & \ldots     & \lambda_{n-2n-2}     &\lambda_{n-2n-1}   & \lambda_{n-2n}  \\
  x^2\lambda_{n-11}    & x^2\lambda_{n-12}   & \ldots     & x\lambda_{n-1n-2}   &\lambda_{n-1n-1}   & \lambda_{n-1n}   \\
  x^2\lambda_{n1}       & x^2\lambda_{n2}      & \ldots     & x^2\lambda_{nn-2}   & x\lambda_{nn-1}   & \lambda_{nn}    
  \end{bmatrix}$$ be an element in $\Lambda$, where $\lambda_{ij} \in R$ for $i,j \in \{1,2, \ldots, n\}$,
    then
  $$\alpha(\lambda) =
    \begin{bmatrix}
  \lambda_{33}             & \lambda_{34}           & \ldots     & \lambda_{3n}        &\lambda_{31}      & x^{-1}\lambda_{32}  \\
  x\lambda_{43}           & \lambda_{44}           & \ldots     & \lambda_{4n}        &\lambda_{41}      & \lambda_{42}   \\
  x^2\lambda_{53}       & x\lambda_{54}          & \ldots     & \lambda_{5n}        &\lambda_{51}      & \lambda_{52}   \\
  \vdots                        & \vdots                       & {\ddots}     & \vdots                   & \vdots                 & \vdots     \\
  x^2 \lambda_{n3}      & x^2\lambda_{n4}      & \ldots     & \lambda_{nn}        &\lambda_{n1}       & \lambda_{n2}  \\
  x^2\lambda_{13}       & x^2\lambda_{14}      & \ldots     & x\lambda_{1n}       &\lambda_{11}       & \lambda_{12}   \\
  x^2\lambda_{23}       & x^2\lambda_{24}      & \ldots     & x^2\lambda_{2n}   & x\lambda_{21}    & \lambda_{22}    
  \end{bmatrix}.$$ 

Let $A$ and $B$ be two $R$-orders. We define ${}_\vartheta M_{\varsigma}$ for an $(A,B)$-bimodule $M$, $\vartheta \in \Aut(A)$ and $\varsigma \in \Aut(B)$ as follows:
${}_\vartheta M_{\varsigma}:=M$ as a vector space and the $(A,B)$-bimodule structure is given by $$a \times  m \times b = \vartheta(a) m \varsigma(b)$$ for $m \in{}_\vartheta M_{\varsigma}$ and $a\in A$, $b \in B$.
Since $\vartheta \in \Aut(A)$, ${}_\vartheta (-)$ is an automorphism of $\mod A$.

Considering the bimodule structures of ${}_1\Lambda_{\alpha}$ and ${}_1(\DR\Lambda)_{\alpha^{-1}}$, we get the following proposition.
 \begin{proposition}
  \label{pro:bimodule}
    The above $f: \Lambda \to \DR \Lambda$ gives an isomorphism of $\Lambda$-bimodules $${}_1\Lambda_{\alpha} \cong \DR\Lambda.$$    \end{proposition}
    \begin{proof}
    Clearly, $f$ preserves the left action of $\Lambda$.
    Moreover, it preserves the right action since for $\lambda$, $ \mu \in \Lambda$, we have $$f(\mu \alpha(\lambda))= f(\mu (P^{-1} \lambda P) )=\mu (P^{-1} \lambda P)  P^{-1}= \mu P^{-1} \lambda=f(\mu) \lambda.$$
  \end{proof}
  
  By using the isomorphism of Proposition \ref{pro:bimodule}, we find the following description of  the Nakayama functor $\nu$.
  \begin{lemma}
  \label{lem:functoriso}
  We have an isomorphism  $\nu \cong {}_{\alpha^{-1}} (-)$ of functors $\underline{\CM}(\Lambda) \to \underline{\CM}(\Lambda)$.
  \end{lemma}
  \begin{proof}
  Since $\DR\Lambda \cong  {}_1\Lambda_{\alpha} $, 
  it follows that $\nu  \cong {}_1\Lambda_{\alpha} \tens_{\Lambda}-$.
On the other hand, we have an isomorphism $H:  {}_1\Lambda_{\alpha} \tens_{\Lambda}- \cong {}_{\alpha^{-1}} (-)$ given by $\lambda \otimes m \mapsto \alpha^{-1}(\lambda)(m)$.
  Thus the assertion follows.
  \end{proof}
   
   Let $T=K[X,Y]$ be a polynomial ring of two variables over the field $K$ and $S=T/(p)$ be the quotient ring with respect to a polynomial $p$.

We define a $\ZZ/n\ZZ$-grading on $T$ by setting $\text{deg}(X)=1$ $(\text{mod }n)$ and $\text{deg}(Y)=-1$ $(\text{mod }n)$. 
   This makes $T$  a $\ZZ/n\ZZ$-graded algebra
   \[
   T=\bigoplus_{\overline{i}\in \ZZ/n\ZZ}T_{\overline{i}}=T_{\overline{0}} \oplus T_{\overline{1}} \oplus \ldots \oplus T_{\overline{n-1}}.
   \]
    Suppose that $p$ is homogeneous of degree $\overline{d}$ with respect to this grading. Then the quotient ring $$S:=T/(p)=S_{\overline{0}} \oplus S_{\overline{1}} \oplus \ldots \oplus S_{\overline{n-1}}$$ has a natural structure of a $\ZZ/n\ZZ$-graded algebra.
    The following result can be easily established from classical results about matrix factorization. For the convenience of the reader, we include a short argument.

\begin{theorem}[\cite{CM}]
\label{degree}
In $\underline{\CM}^{\ZZ/n\ZZ}(S)$, there is an isomorphism of autoequivalences $[2] \cong (-d)$.
\end{theorem}

\begin{proof}
For $M \in \CM^{\ZZ/n\ZZ}(S)$, using a similar argument than in \cite[Proposition 7.2]{CM}, the Auslander-Buchsbaum formula
$$\prdim_T(M) = \depth(T) - \depth_T(M) = 1$$ 
induces the existence of graded $T$-modules 
$$U = \bigoplus_{i = 1}^k T(j_i) \quad \text{and} \quad V = \bigoplus_{i = 1}^k T(j'_i),$$ 
$B \in \Hom_{\mod^{\ZZ/n\ZZ} T} (U(d), V)$, $C \in \Hom_{\mod^{\ZZ/n\ZZ} T} (V, U)$ and $\pi \in \Hom_{\mod^{\ZZ/n\ZZ} T} (U, M)$ such that
$$0 \rightarrow V \xrightarrow{C} U \xrightarrow{\pi} M \rightarrow 0$$
is a short exact sequence of graded $T$-modules, $C B =p \Id_U:U(d) \rightarrow U$, $B C(d) =p \Id_V:V(d) \rightarrow V$ and 
\begin{align*} \cdots \xrightarrow{C(2d) \tens S} & U(2d) \tens_T S \xrightarrow{B(d) \tens S} V(d) \tens_T S \xrightarrow{C(d) \tens S} U(d) \tens_T S \\ & \xrightarrow{B\tens S} V\tens_T S \xrightarrow{C\tens S} \bar U\tens S \xrightarrow{\pi} M \rightarrow 0\end{align*}
is an exact sequence in $\CM^{\ZZ/n\ZZ}(S)$ (the $d$-shifts come from the fact that $C B$ and $B C$ are homogeneous of degree $d$). Then, we conclude that $\Omega^2(M)$ is functorially isomorphic to $\coker (C(d) \tens S) \cong M(d)$ in $\underline{\CM}^{\ZZ/n\ZZ}(S)$.
\end{proof}

Setting $p:=X^{n-2}-Y^2$, we have $S=T/(X^{n-2}-Y^2)$. 
Identifying $R=K[x]$ to the subalgebra $K[XY]$ of $S$ via $x \mapsto XY$, we regard $S$ as an $R$-algebra and we get the following lemma.
   
\begin{lemma}
\label{lem:matrixS}
We have $$S_{\overline{i}}
=\begin{cases}
  R \overline{X}^i, &\text{if } i\in \{0,1,\ldots, n-2\}, \\
  R \overline{Y}, &\text{if } i=n-1.
\end{cases}
$$
\end{lemma}
\begin{proof}
   Any element in $S$ can be written as $\sum_{c,d}\lambda_{c,d} \overline{X^c Y^d}$, where $c$, $d$ run over the set of non-negative integers and $\lambda_{c,d} \in K$.
   Let $i \in \{0,1, \dots, n-2\}$. For an element $\sum_{c,d} \lambda_{c,d} \overline{X^c Y^d}$ in $S_{\overline{i}}$, we have $\text{deg}(\overline{X^c Y^d})=i$ $(\text{mod } n)$, which means $c-d=i$ $(\text{mod }n)$. 
 So $c=d+i+tn$ holds for some integer $t$. If $t\ge 0$, then $$\overline{X^c Y^d} = \overline{X^{d+i+tn} Y^d} = \overline{(XY)^{d+2t} X^i}.$$
  If $t<0$, then 
  \begin{align*}
  \overline{X^c Y^d} = \overline{X^c Y^{c-i-tn}}  = \overline{(XY)^{c-t(n-2)-i} X^i}.
  \end{align*}
   Therefore, $S_{\overline{i}} = K[\overline{XY}]\overline{X}^i$. In particular, the degree $0$ part $S_{\overline{0}} = K[\overline{XY}]$ is isomorphic to $R$ as a ring.
   Similarly, we have $S_{\overline{n-1}} = K[\overline{XY}]\overline{Y}$.
   Hence, we have $S_{\overline{i}} \cong R$ as $R$-module for each $\overline{i}\in \ZZ/n\ZZ$.
 \end{proof}
 We define the $R$-order $S^{[n]}$ as a subalgebra of $\M_n(S)$ as follows:  
 $$S^{[n]}=\begin{bmatrix}
      S_{\overline{0}} & S_{\overline{1}}  & S_{\overline{2}}    &\cdots & S_{\overline{n-2}}    & S_{\overline{n-1}}\\
      S_{\overline{n-1}} & S_{\overline{0}}  & S_{\overline{1}}    &\cdots & S_{\overline{n-3}}    & S_{\overline{n-2}}\\
      S_{\overline{n-2}} & S_{\overline{n-1}}  & S_{\overline{0}}    &\cdots & S_{\overline{n-4}}    & S_{\overline{n-3}}\\
      \vdots           & \vdots            & \vdots              &\ddots & \vdots              & \vdots\\
      S_{\overline{2}} & S_{\overline{3}}  & S_{\overline{4}}    &\cdots & S_{\overline{0}}    & S_{\overline{1}}\\
      S_{\overline{1}} & S_{\overline{2}}  & S_{\overline{3}}    &\cdots & S_{\overline{n-1}}    & S_{\overline{0}}
  \end{bmatrix}.$$
  
    \begin{proposition}
    \label{pro:iso}
We have an isomorphism $S^{[n]} \to \Lambda$ of $R$-algebras.
  \end{proposition}
  \begin{proof}
According to Lemma \ref{lem:matrixS}, we have
   $$S^{[n]}=\begin{bmatrix}
      R & R\overline{X}  & R\overline{X}^2    &\cdots & R\overline{X}^{n-2}   & R\overline{Y}\\
      R\overline{Y} & R  & R\overline{X}    &\cdots & R\overline{X}^{n-3}    & R\overline{X}^{n-2}\\
      R\overline{X}^{n-2} & R\overline{Y}  & R    &\cdots & R\overline{X}^{n-4}    & R\overline{X}^{n-3}\\
      \vdots           & \vdots            & \vdots              &\ddots & \vdots              & \vdots\\
      R\overline{X}^{2} & R\overline{X}^{3}  & R\overline{X}^4    &\cdots & R  & R\overline{X}\\
      R\overline{X}^{1} & R\overline{X}^{2}  &R\overline{X}^{3}    &\cdots & R\overline{Y}    & R
  \end{bmatrix}.$$
  Taking the conjugation by 
  $$B=\diag\left(\overline{X}^i\right)_{1 \leq i \leq n,}$$ we get easily $B S^{[n]} B^{-1} = \Lambda$ via the above identification $R \cong K[XY]$.
  \end{proof}
From now on, we identity $\Lambda$ and $S^{[n]}$.  
 Consider the matrix 
 $$G'=
   \begin{bmatrix}
  0           & 0            & 0            & \ldots     & 0              & 1              & 0   \\
  0           & 0            & 0            & \ldots     & 0              & 0              & 1   \\
   1      & 0            & 0            & \ldots     & 0              & 0              & 0    \\
   0          & 1        & 0            & \ldots     & 0              & 0              & 0   \\
  \vdots   & \vdots   & \vdots    & \vdots    & \vdots      & \vdots     & \vdots \\
  0           & 0            & \ldots     & 0            & 1         & 0              & 0
  \end{bmatrix}.$$
The automorphism $\beta$ of $S^{[n]}$ given by $\beta(s)=G'^{-1}s G'$ for $s \in S^{[n]}$ corresponds to the automorphism $\alpha$ of $\Lambda$.
  Thus we have an isomorphism ${}_1S^{[n]}_{\beta^{-1}} \rightarrow  {}_1 \Lambda_{\alpha^{-1}}$ of $S^{[n]}$-bimodules.

Using the same notation above, we have the following easy lemma (see \cite[Theorem 3.1]{IyamaLerner} for details in a wider context).
\begin{lemma}
\begin{enumerate}
 \item 
 The functor 
\begin{align*}
F:\mod^{\ZZ/n\ZZ}S &\to \mod S^{[n]}\\
M_{\overline{0}} \oplus M_{\overline{1}} \oplus \ldots \oplus M_{\overline{n-1}} 
& \mapsto 
  \begin{bmatrix}
  M_{\overline{0}} & M_{\overline{1}} & \ldots & M_{\overline{n-1}}
  \end{bmatrix}^{\operatorname{t}}
\end{align*} is an equivalence of categories.
\item For $i\in \ZZ$, we denote by
$(i):\mod^{\ZZ/n\ZZ}S \to \mod^{\ZZ/n\ZZ}S$
the grade shift functor defined by $M(i)_{\overline{j}}:=M_{\overline{i+j}}$ for $M \in \mod^{\ZZ/n\ZZ}S$.
The grade shift functor $(i)$ induces an autofunctor (denoted by $\gamma_i$) in $\mod S^{[n]}$ which makes the following diagram commute:
\[
  \begin{tikzcd}
  \mod^{\ZZ/n\ZZ}S \arrow{r}{F} \arrow{d}[swap]{(i)} &\mod S^{[n]} \arrow{d}{\gamma_i}\\
  \mod^{\ZZ/n\ZZ}S \arrow{r}{F} & \mod S^{[n]}.
  \end{tikzcd}
\]
\end{enumerate}
\end{lemma}

More precisely, for any left $S^{[n]}$-module $\begin{bmatrix}
  M_{\overline{0}} & M_{\overline{1}} & \ldots & M_{\overline{n-1}}
  \end{bmatrix}^{\operatorname{t}}$, we have $$\gamma_i \left(\begin{bmatrix}
 M_{\overline{0}} & M_{\overline{1}} & \ldots & M_{\overline{n-1}}
  \end{bmatrix}^{\operatorname{t}}\right)=\begin{bmatrix}
  M_{\overline{i}} & M_{\overline{i+1}} & \ldots & M_{\overline{i+n-1}}
  \end{bmatrix}^{\operatorname{t}}.$$
  
Now we can prove the $2$-Calabi-Yau property of $\underline{\CM}(\Lambda)$.
\begin{proof}[Proof of Theorem \ref{th:cyeq}~(1)]
    The equivalence $\mod^{\ZZ/n\ZZ}S \cong \mod S^{[n]} = \mod\Lambda$ induces an equivalence $$\CM^{\ZZ/n\ZZ}S \cong \CM S^{[n]} = \CM\Lambda.$$
     In the category $\underline{\CM}^{\ZZ/n\ZZ}S$, according to Theorem \ref{degree}, we have an isomorphism of functors $$[2] \cong (-\text{deg }(x^{n-2}-y^2))=(2).$$
     By Lemma \ref{lem:functoriso}, we have $\nu \cong {}_{\alpha^{-1}}(-)$. Therefore, it is enough to prove ${}_{\alpha^{-1}} (-) \cong (2)$.
     This is equivalent to prove that ${}_\beta M \cong \gamma_{-2}(M)$ holds for any $M \in \CM S^{[n]}$.
    
 Let $s_i$ be the row matrix which has $1$ in the $i$-th column and $0$ elsewhere. We have
   \begin{align*}
     {}_\beta M & \cong     \begin{bmatrix}
       s_{0} \times {}_\beta M & s_{1}\times {}_\beta M & s_2 \times {}_\beta M & \ldots & s_{n-1}\times {{}_\beta M}
      \end{bmatrix}^{\operatorname{t}}\\
     & =
     \begin{bmatrix}
     \beta(s_{0})M & \beta(s_{1})M & \beta(s_{2})M & \ldots & \beta(s_{n-1})M
     \end{bmatrix}^{\operatorname{t}}\\
     & =
     \begin{bmatrix}
     s_{n-2}M & s_{n-1}M & s_{0} M & \ldots & s_{n-3}M
     \end{bmatrix}^{\operatorname{t}} \cong \gamma_{-2}(M).
   \end{align*}
   Therefore, the category $\underline{\CM}(\Lambda)$ is $2$-Calabi-Yau.
  \end{proof}

\subsection{Application to cluster algebras}
\label{ss:Application to cluster algebras}

As an application of the previous results, we categorify the cluster algebra structure (including coefficients) on homogeneous coordinate ring of the Grassmannian of $2$-dimensional planes in $L^n$ for any field $L$. We suppose that $K$ is algebraically closed (note that $K$ and $L$ are not necessarily equal). As before, $\mathcal{S}$ is the set consisting of sides and diagonals of the polygon $P$ with $n$ vertices. Recall that the \emph{homogeneous coordinate ring of the Grassmannian of $2$-dimensional planes in $L^n$} can be presented in the following way:
$$L[\Gr_2(L^n)] = \frac{L[\Delta_{(s,t)}]_{(s,t) \in \mathcal{S}}}{I}$$
where $I$ is the ideal generated by the relations of the form $\Delta_{(p,s)} \Delta_{(q,t)} = \Delta_{(p,q)} \Delta_{(s,t)} + \Delta_{(p,t)} \Delta_{(q,s)}$ with $1 \leq p < q < s < t \leq n$. Notice that these relations relate the product of two crossing diagonals with the products of opposite sides of the corresponding quadrilateral. Moreover, Fomin and Zelevinsky \cite{FoZe02} defined on $L[\Gr_2(L^n)]$ a so-called \emph{cluster algebra structure}, whose cluster variables are the \emph{Pl\"ucker coordinates} $\Delta_{(s,t)}$ for $(s,t) \in \mathcal{S}$, and whose \emph{clusters} consist of the sets of Pl\"ucker coordinates corresponding to triangulations of $P$. The \emph{mutations} of clusters correspond naturally to the flips of triangulations and the \emph{exchange relations} to the former relations between Pl\"ucker coordinates.

We will use results of Fu and Keller \cite[section 3]{FuKeller} to prove that $\CM(\Lambda)$ \emph{categorifies} this cluster algebra structure. We fix a triangulation $\sigma$ of $P$. For $M \in \CM(\Lambda)$, following Fu and Keller and using their notations, we consider the following Laurent polynomial over $L$ in $2n-3$ variables $(x_i)_{i \in \sigma}$:
$$X'_M = \prod_{i \in \sigma} x_i^{\left \langle \Hom_{\Lambda}(T_\sigma, M), S_i \right\rangle_\tau} \sum_{e \in \NN^\sigma} \chi\left( \Gr_e (\Ext^1_{\Lambda}(T_\sigma, M))\right) \prod_{i \in \sigma} x_i^{-\langle e, S_i \rangle_3} \in L\left[x_i^{\pm 1}\right]_{i \in \sigma}$$
where $\Hom_{\Lambda}(T_\sigma, M)$ and $\Ext^1_{\Lambda}(T_\sigma, M)$ are seen as $\Gamma_\sigma$-modules using Theorem \ref{th:endomorphism}, $S_i$ is the simple top of $\Gamma_\sigma e_i$ and we define for a finitely generated $\Gamma_\sigma$-module and $i \in \sigma$, 
\begin{align*} &\langle X, S_i \rangle_\tau = \dim \Hom_{\Gamma_\sigma} (X, S_i) - \dim \Ext^1_{\Gamma_\sigma} (X, S_i) \\ \text{and} \quad &\langle X, S_i \rangle_3 = \sum_{i = 0}^3 (-1)^i \dim \Ext^i_{\Gamma_\sigma} (X, S_i).\end{align*}
As $\gldim \Gamma_\sigma \leq 3$, $\langle X, S_i \rangle_3$ depends only on $\underline{\dim}(X) \in \NN^\sigma$. Finally, $\chi( \Gr_e (\Ext^1_{\Lambda}(T_\sigma, M)))$ is the Euler characteristic of the Grassmannian of $\Gamma_\sigma$-submodules of $\Ext^1_{\Lambda}(T_\sigma, M)$ of dimension vector $e$ (the Euler characteristic is the one of the $l$-adic cohomology). The following theorem is a slight generalization of \cite[Theorem 3.3]{FuKeller} and can be proved in the same way.

\begin{theorem}[after {\cite[Theorem 3.3]{FuKeller}}]
 \label{th:FK}
 \begin{enumerate}
  \item For $i \in \sigma$, $X'_{e_F \Gamma_\sigma e_i} = x_i$.
  \item For any $M, N \in \CM(\Lambda)$, $X'_{M \oplus N} = X'_M X'_N$.
  \item If $M, N \in \CM(\Lambda)$ satisfy $\Ext^1_{\Lambda}(M, N) = 1$ and 
   $$0 \rightarrow M \rightarrow E \rightarrow N \rightarrow 0 \quad \text{and} \quad 0 \rightarrow N \rightarrow E' \rightarrow M \rightarrow 0$$
   are non-split short exact sequences then $X'_M X'_N = X'_E + X'_{E'}$.
 \end{enumerate}
\end{theorem}

Therefore, the map $X'$ is called a \emph{cluster character}. 

\begin{remark}
 The only difference with the setting of \cite{FuKeller} is that $\CM(\Lambda)$ does not have finite dimensional morphism spaces over $K$. However, the considered morphism spaces are finitely generated as $\Gamma_\sigma$-modules and $\underline{\CM}(\Lambda)$ has finite dimensional morphism spaces over $K$. In particular, $X'$ is well defined. Moreover, under these conditions, it is easy to check that the proofs in \cite{FuKeller} still apply.
\end{remark}

Before stating the main result of this section, let us do an easy observation ($\sigma$ is still a fixed triangulation):

\begin{lemma} \label{lem:obs}
 The cluster algebra with initial seed $(\{X'_{e_F \Gamma_\sigma e_i} \,|\, i \in \sigma\}, (Q_\sigma, F))$ coincides with the $L$-vector space spanned by $\{X'_M \,|\, M \in \CM(\Lambda)\}$ in $L\left[x_i^{\pm 1}\right]_{i \in \sigma}$. 
\end{lemma}

\begin{proof}
 Thanks to Theorem \ref{th:FK}, the cluster algebra is generated by $X'_M$ for all $M \in \CM(\Lambda)$ rigid such that $M$ is a summand of a cluster tilting object reachable from $T_\sigma$ by sequences of mutations. Here, all indecomposable objects are summand of such cluster tilting objects (they correspond to the diagonals of the polygon, which are all part of some triangulation reachable by sequences of flips). Thus, using Theorem \ref{th:FK}~(2), we get the result.
\end{proof}

We deduce from this theorem the following categorification:

\begin{theorem} \label{th:grassman}
 There is an isomorphism of cluster algebras (\emph{i.e.} an isomorphism of $L$-algebras mapping clusters to clusters and compatible with the mutation)
 \begin{align*}
  \kappa: L[\Gr_2(L^n)] & \rightarrow \sum_{M \in \CM(\Lambda)}  L X'_M \subset L[x_i^{\pm 1}]_{i \in \sigma}\\
   \Delta_{(s,t)} & \mapsto X'_{(s,t)} \quad \text{for } (s,t) \in \mathcal{S}
 \end{align*}
 where the cluster algebra structure of the right member is the one categorified in \cite{FuKeller} (see Lemma \ref{lem:obs}).
\end{theorem}

\begin{proof}
 First of all, consider a relation $\Delta_{(p,s)} \Delta_{(q,t)} = \Delta_{(p,q)} \Delta_{(s,t)} + \Delta_{(p,t)} \Delta_{(q,s)}$ with $1 \leq p < q < s < t \leq n$. Thanks to Lemma \ref{pro:extension}, we have the two following non-split short exact sequences in $\CM(\Lambda)$:
 \begin{align} & 0\lra (p,s) \lra (p,t)\oplus(q,s) \lra (q,t)\lra 0 \notag \\
 \text{and} \quad &0\lra (q,t) \lra (p,q)\oplus(s,t) \lra (p,s)\lra 0. \label{excseq} \end{align}
 Moreover, using the explicit description of $\Ext^1_{\Lambda}$ or the description of the cluster category of type $A_{n-3}$ and Theorem \ref{th:cyeq}, we know that $\dim_K \Ext^1_{\Lambda}((p,s), (q,t)) = 1$. Therefore, thanks to Theorem \ref{th:FK}~(3), we get 
  $$X'_{(p,s)} X'_{(q,t)} = X'_{(p,q)} X'_{(s,t)} + X'_{(p,t)} X'_{(q,s)}$$
  and therefore the morphism is well defined. The surjectivity is immediate from Lemma \ref{lem:obs}. As the fraction field of $L[\Gr_2(L^n)]$ is the rational functions field $L(\Delta_i)_{i \in \sigma}$, $\kappa$ is injective.

 The compatibility with the cluster structure is immediate (by definition, the cluster structure on $\sum_{M \in \CM(\Lambda)}  L X'_M$ is determined by the cluster tilting objects and the exchange relations come from approximations sequences \eqref{excseq}).
\end{proof}

\section{Graded Cohen-Macaulay $\Lambda$-modules}
\label{s:graded}
In this section, we study a graded version of Theorem \ref{th:cyeq} giving a relationship between the category of Cohen-Macaulay $\Lambda$-modules and the cluster category of type $A_{n-3}$.
We study the category $\CM^{\ZZ}(\Lambda)$ of graded Cohen-Macaulay $\Lambda$-modules and its relationship with the bounded derived category $\cd^{\operatorname{b}}(KQ)$ of type $A_{n-3}$.

Let $Q$ be an acyclic quiver. 
We denote by $\ck^{\operatorname{b}}(\proj KQ)$ the bounded homotopy category of finitely generated projective $KQ$-modules, and by $\cd^{\operatorname{b}}(KQ)$ the bounded derived category of finitely generated $KQ$-modules.
These are triangulated categories and the canonical embedding 
$\ck^{\operatorname{b}}(\proj KQ) \to \cd^{\operatorname{b}}(KQ)$ is a triangle functor.

We define a grading on $\Lambda$ by $\Lambda_i = \Lambda \cap \M_n(K x^i)$ for $i \in \ZZ$.
 This makes $\Lambda=\bigoplus_{i\in \ZZ}\Lambda_i$ a $\ZZ$-graded algebra.
The category of graded Cohen-Macaulay $\Lambda$-modules, $\CM^{\ZZ}(\Lambda)$, is defined as follows.
The objects are graded $\Lambda$-modules which are Cohen-Macaulay,
and the morphisms in $\CM^{\ZZ}(\Lambda)$ are $\Lambda$-morphisms preserving the degree.
The category $\CM^{\ZZ}(\Lambda)$ is a Frobenius category. Its stable category is denoted by $\underline{\CM}^{\ZZ}(\Lambda)$.
For $i\in \ZZ$, we denote by
$(i):\CM^{\ZZ}(\Lambda) \to \CM^{\ZZ}(\Lambda)$
the grade shift functor: Given a graded Cohen-Macaulay $\Lambda$-module $X$, we define $X(i)$ to be $X$ as a $\Lambda$-module, with the grading $X(i)_j = X_{i+j}$ for any $j \in \ZZ$. 

\begin{remark}
We show that this grading of $\Lambda$ is analogous to the grading of $\Lambda_\sigma$ given by the $\theta$-length.
 Let $i,j\in F$. By Theorem \ref{thm:frozen part}, $e_i \Lambda_\sigma e_j \cong e_i \Lambda e_j$ holds.
Let $\lambda \in e_i \Lambda_\sigma e_j \cong e_i \Lambda e_j$. According to Theorem \ref{cor:tiled order}, 
we have $${\frac{\ell^{\theta}(\lambda)+\ell^{\theta}_{1,i}  -\ell^{\theta}_{1,j}  +n\delta_{i,n} - n\delta_{j,n}}{n}}={\deg}(\lambda).$$
Consider the two graded algebras $$\Lambda_\sigma =\bigoplus_{i=1}^n \Lambda_\sigma e_i \and \Lambda' := \End\left(\bigoplus_{i=1}^n \Lambda_\sigma e_i \left(\ell^{\theta}_{1,i} +n\delta_{i,n}\right) \right).$$ By graded Morita equivalence, we have $\CM^\ZZ(\Lambda')\cong \CM^\ZZ(\Lambda_\sigma)$.
Since $\Lambda \cong \Lambda'$ as $R$-orders and $deg(x)=n$ in $\Lambda'$, it follows that the Auslander-Reiten quiver of $\CM^\ZZ(\Lambda')$ has $n$ connected components each of which is a degree shift of the Auslander-Reiten quiver of $\CM^\ZZ(\Lambda)$.
\end{remark}

The exact duality $\DR: \CM(\Lambda^{\operatorname{op}}) \to \CM(\Lambda)$ induces naturally an exact duality $\DR: \CM^{\ZZ}(\Lambda^{\operatorname{op}}) \to \CM^{\ZZ}(\Lambda)$.
In the same way, $\Omega$ and $\tr$ can be enhanced as functors $\Omega:\underline{\CM}^{\ZZ}(\Lambda^{\operatorname{op}}) \to \underline{\CM}^{\ZZ}(\Lambda^{\operatorname{op}})$ and $\tr:  \underline{\CM}^{\ZZ}(\Lambda) \to \underline{\CM}^{\ZZ}(\Lambda^{\operatorname{op}})$. Thus, as before we denote $\tau= \DR \Omega \tr :\underline{\CM}^{\ZZ}(\Lambda) \to \underline{\CM}^{\ZZ}(\Lambda) $.
We introduce the properties of $ \CM^\ZZ(\Lambda)$ in the following theorems.
\begin{theorem}[{\cite[Theorem 1.1]{ARgraded}}]
\label{th:gradedtau}
 Let $A$ be a $\ZZ$-graded $R$-order. Let the degree of $x$ be $d$. If there is an Auslander-Reiten sequence $0\lra L \lra M \lra N \lra0$ in $\CM^\ZZ(A)$, then $L=\tau N(-d)$.
\end{theorem}

\begin{theorem}
\label{th:graded modules}
\begin{enumerate}
   \item The set of isomorphism classes of indecomposable graded Cohen-Macaulay $\Lambda$-modules is $$\{(i,j) \mid i,j \in \ZZ, 0 < j-i < n\},$$ where 
  $$(i,j):=[\overbrace{R \cdots R}^{i} \overbrace{(x)  \cdots (x)}^{j-i} \overbrace{(x^2)  \cdots (x^2)}^{n-j}]^t \for 1\le i<j\le n,$$
 $$(i+kn,j+kn):=(i,j)(2k) \and (j+kn,i+(k+1)n):=(i,j)(2k+1)$$ for $k\in \ZZ$. The projective-injective objects are of the form $(i, i+1)$ and $(i, i+n-1)$ for $i \in \ZZ$.
 \item The non-split extensions of graded indecomposable Cohen-Macaulay $\Lambda$-modules are of the form 
$$0\lra (s,t) \lra (s,t')\oplus(s',t) \lra (s',t')\lra 0$$
with $s<s'<t<t'<s+n$.
 \item For any non-projective indecomposable graded Cohen-Macaulay $\Lambda$-module $(i,j)$, the Auslander-Reiten sequence ending with $(i,j)$ is of the following form: $$0\lra (i-1,j-1) \lra (i-1,j)\oplus(i,j-1) \lra (i,j)\lra 0.$$
 \item The Auslander-Reiten quiver of $\CM^{\ZZ}(\Lambda)$ is the following:
{\tiny 
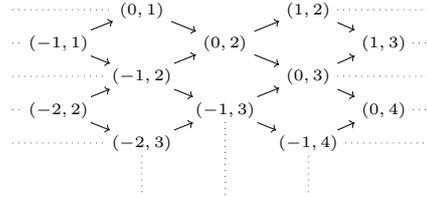
\[
\begin{tikzcd}[column sep=tiny, row sep=tiny]
\arrow[dotted,-]{rr}&&(0,1) \drar & &(1,2)\drar\arrow[dotted,-]{rrr}&&&\phantom{X}\\
\arrow[dotted,-]{r}&(-1,1)\urar\drar & &(0,2)\urar\drar&&(1,3)\arrow[dotted,-]{rr} &&\phantom{X}\\
\arrow[dotted,-]{rr}&&(-1,2)\urar\drar&&(0,3)\urar\drar\arrow[dotted,-]{rrr}& && \phantom{X}\\ 
\arrow[dotted,-]{r}&(-2,2)\urar\drar & &(-1,3)\arrow[dotted,-]{ddd}\urar\drar &&(0,4) \arrow[dotted,-]{rr} &&\phantom{X}\\
\arrow[dotted,-]{rr}&&(-2,3)\arrow[dotted,-]{dd} \urar&&(-1,4)\arrow[dotted,-]{dd}\urar \arrow[dotted,-]{rrr}&&&\phantom{X}\\
&&\phantom{X}&\phantom{X}&\phantom{X}&&&\phantom{X}\\
&&\phantom{X}&\phantom{X}&\phantom{X}&&&\phantom{X}
\end{tikzcd}
\]}
 \end{enumerate}
\end{theorem}
\begin{proof}
(1) First of all, it is immediate that the graded modules $(i,j)$ for $0 < j-i < n$ are not isomorphic. Therefore, we need to prove that there are no other isomorphism classes. We consider the degree forgetful functor $F: \CM^\ZZ(\Lambda) \to \CM(\Lambda)$. Let $X \in \CM^\ZZ(\Lambda)$ be indecomposable and $(i,j)$ be an indecomposable summand of $FX$ in $\CM(\Lambda)$. There are two morphisms $f: (i,j) \rightarrow F X$ and $g: F X \rightarrow (i,j)$ such that $gf = \Id_{(i,j)}$. Let us write 
$$f = \sum_{m \in \ZZ} f_m \and g = \sum_{m \in \ZZ} g_m$$
where $f_m$ is a graded morphism from $(i,j)$ to $X(-m)$ and $g_m$ a graded morphism from $X(m)$ to $(i,j)$. Thus, we have
$$\sum_{k \in \ZZ} g_{k} f_{-k} = \Id_{(i,j)}$$
and, as the graded endomorphism ring of $(i,j)$ is $K$, there exists $k \in \ZZ$ such that $g_k f_{-k}$ is a nonzero multiple of $\Id_{(i,j)}$. In other terms we found two graded morphisms $\tilde f: (i,j) \rightarrow X(k)$ and $\tilde g: X(k) \rightarrow (i,j)$ such that $\tilde g \tilde f = \Id_{(i,j)}$. Thus, as idempotents split in $\CM^\ZZ(\Lambda)$ and $X$ is indecomposable, we get that $X \simeq (i,j)(-k)$.
Therefore, the set of isomorphism classes of indecomposable graded Cohen-Macaulay $\Lambda$-modules is $\{(i,j) \mid {i,j\in \ZZ, 0 < j-i < n}\}$.

(2) We omit the proof since it is analogous to Lemma \ref{pro:extension}.

(3) Consider the non-projective indecomposable graded Cohen-Macaulay $\Lambda$-module $(i,j)$. Thanks to (2), there is a short exact sequence
\begin{equation}
\label{eq:projective cover}
 0\lra (i+1,j+1)(-1) = (j+1-n,i+1) \lra (j+1-n, j) \oplus (i,i+1) \lra (i,j)\lra 0.
\end{equation}
in $\CM^\ZZ(\Lambda)$. Therefore, as $(i,i+1)\oplus (j+1-n, j)$ is projective and $(i,j)$ is non-projective, $(i+1,j+1)(-1)$ is the syzygy of $(i,j)$. Apply $\Hom_{\Lambda}(-,\Lambda)$ to \eqref{eq:projective cover}, when $i=1$ we get the following short exact sequence in $\CM^\ZZ({\Lambda}^{\operatorname{op}})$:
\begin{align*}0 \ra [\overbrace{ (x)\cdots (x)}^{j-2} \overbrace{R  \cdots R}^{n-j+1} (x^{-1})] &\ra [\overbrace{R \cdots R}^{n-1}(x^{-1})] \oplus [\overbrace{ (x)\cdots (x)}^{j-2} R \overbrace{(x^{-1}) \cdots (x^{-1})}^{n-j+1}] \\ &\ra [\overbrace{R \cdots R}^{j-1}\overbrace{(x^{-1}) \cdots (x^{-1})}^{n-j+1}] \ra 0,\end{align*}
and when $i>1$ we have the following short exact sequence in $\CM^\ZZ({\Lambda}^{\operatorname{op}})$:
\begin{align*}0 \ra [{\overbrace{ (x^2)\cdots (x^2)}^{i-2}}\overbrace{ (x)\cdots (x)}^{{j-i}} \overbrace{R  \cdots R}^{n-j+2}] &\ra [{\overbrace{(x^2) \cdots (x^2)}^{i-2}(x)}\overbrace{R \cdots R}^{n-i+1}] \oplus [\overbrace{ (x)\cdots (x)}^{j-2} R \overbrace{(x^{-1}) \cdots (x^{-1})}^{n-j+1}] \\ &\ra [\overbrace{(x) \cdots (x)}^{i-1}\overbrace{R \cdots R}^{j-i}\overbrace{(x^{-1}) \cdots (x^{-1})}^{n-j+1}] \ra 0. \end{align*}
Therefore, according to Theorem \ref{th:gradedtau}, when $i=1$ we have $$\tau(1,j)(-1)=\DR([\overbrace{R \cdots R}^{j-1}\overbrace{(x^{-1}) \cdots (x^{-1})}^{n-j+1}])(-1)=(j-1,n)(-1)=(0,j-1),$$
 when $i>1$ we have $$\tau(i,j)(-1)=\DR([\overbrace{(x) \cdots (x)}^{i-1}\overbrace{R \cdots R}^{j-i}\overbrace{(x^{-1}) \cdots (x^{-1})}^{n-j+1}])(-1)=(i-1,j-1).$$
 By the previous point, the only non-split extension from $(i,j)$ to $(i-1,j-1)$ (up to isomorphism) is $$0\lra (i-1,j-1) \lra (i-1,j)\oplus(i,j-1) \lra (i,j)\lra 0,$$ so it is an Auslander-Reiten sequence.
 
 (4) This is a direct consequence of (1) and (3).
\end{proof}

\begin{definition}
Let $\cc$ be a triangulated category. An object $T$ is said to be \emph{tilting} if $\Hom_\cc(T, T[k])=0$ for any $k \neq 0$ and $\thick(T)=\cc$, where $\thick(T)$ is the smallest full triangulated subcategory of $\cc$ containing $T$ and closed under isomorphisms and direct summands.
\end{definition}

\begin{theorem}[{\cite[Theorem 4.3]{derdgcat}, \cite[Theorem 2.2]{IyTa13}, \cite{BK1991}}]
\label{BK}
Let $\cc$ be an algebraic triangulated Krull-Schmidt category. If $\cc$ has a tilting object $T$, then there exists a triangle-equivalence
$$\cc \rightarrow \ck^{\operatorname{b}}(\proj \End_{\cc}(T)).$$
\end{theorem}

Now we get the following theorem which is analogous to Theorem \ref{thm:cluster tilting} and Theorem \ref{th:equivalence}.
\begin{theorem}
  \label{thm:cluster cat as stable cat for A type}
 Let $\Lambda$ be the graded $R$-order defined as above and $Q$ be a quiver of type $A_{n-3}$ with $n\ge3$. Then
  \begin{enumerate}
     \item for a triangulation $\sigma$ of the polygon $P$, the Cohen-Macaulay $\Lambda$-module $e_F \Gamma_\sigma$ can be lifted to a tilting object in $\underline{\CM}^\ZZ(\Lambda)$;
     \item then there exists a triangle-equivalence $\cd^{\operatorname{b}}(KQ) \cong \underline{\CM}^{\ZZ}(\Lambda)$.
  \end{enumerate}
  \end{theorem}

\begin{proof}
(1) First, we have $e_F\Gamma_\sigma\cong \bigoplus_{(i,j)\in \sigma}(i,j)$. So we need to choose some degree shift of each $(i, j)$ appearing in $\sigma$. There exists a vertex $i_0$ of $P$ such that $i_0$ does not have any incident internal edge in $\sigma$. For each $(i,j) \in \sigma$, denote 
$$(i,j)' = \begin{cases} (i,j), & \text{if } j < i_0, \\ (i,j)(-1) = (j-n, i), & \text{if } j > i_0. \end{cases} $$
and $\sigma' = \{(i,j)' \,|\, (i,j) \in \sigma\}$. Let us prove that the graded module $T'_{\sigma'} = \bigoplus_{(i,j) \in \sigma'} (i,j)$ is tilting (it is $T_\sigma$ if we forget the grading). Notice first that for any $(i,j) \in \sigma'$, we have $i_0-n < i < j < i_0$. Let us check that  $$\Hom_{\underline{\CM}^{\ZZ}(\Lambda)}((s,t), \Omega^k (i,j))=0$$ for any $(i,j), (s,t)\in \sigma'$ and $k \neq 0$. Thanks to \eqref{eq:projective cover} in the proof of Theorem \ref{th:graded modules}, we know that
$$\Omega^{k}(i,j)=(i+k,j+k)(-k)$$ 
and then, for $\ell \in \ZZ$,
\begin{align*}\Omega^{2\ell}(i,j)&=(i+(2-n)\ell,j+(2-n)\ell) \\ \text{and} \quad \Omega^{2\ell+1}(i,j)&=(j+(2-n)(\ell+1)-1,i+(2-n)\ell+1).\end{align*}
Therefore, denoting $(i',j') = \Omega^k(i,j)$, if $k > 0$, we get $j' \leq j + 2 - n \leq i_0 + 1 - n \leq s < t$ or $i' \leq j+1-n \leq i_0 - n < s$ so 
$$\Hom_{\underline{\CM}^{\ZZ}(\Lambda)}((s,t), \Omega^k (i,j))=0.$$
If $k < -1$, we get $i' \geq i + n - 2 \geq i_0 - 1 \geq t$ and $j' \geq j + n - 2 \geq t+1$, so any morphism from $(s,t)$ to $(i',j')$ in $\CM^\ZZ(\Lambda)$ factors through $(t-1,t)$ which is projective.

For $k = -1$, by Theorem \ref{thm:cluster tilting}, we get $$\Hom_{\underline{\CM}^{\ZZ}(\Lambda)}(T'_{\sigma'}, \Omega^{-1} T'_{\sigma'}) {\subset}\Hom_{\underline{\CM}(\Lambda)}(T_{\sigma},  \Omega^{-1} T_{\sigma})= 0.$$

Let us now prove that 
$$\thick(T'_{\sigma'})=\underline{\CM}^{\ZZ}(\Lambda).$$

Consider an object of the form $(s,i_0) \in \CM^\ZZ(\Lambda)$. Let $(s', t') \in \CM(\Lambda)$ be its image through the forgetful functor. There exists a short exact sequence $$0 \rightarrow T'_1 \rightarrow T'_0 \rightarrow (s', t') \rightarrow 0$$ of Cohen-Macaulay $\Lambda$-modules such that $T'_0, T'_1 \in \add(T_\sigma)$. Indeed, let $\tilde f: \tilde T'_0 \rightarrow (s',t')$ be a right $\add(T_\sigma)$-approximation of $(s',t')$ (\emph{i.e.} $\tilde T'_0 \in \add(T_\sigma)$ and every morphism $g: T_\sigma \rightarrow (s',t')$ factors through $\tilde f$). As $\CM(\Lambda)$ is an exact category, and 
$$0 \rightarrow  (s'+1, t'+1) \rightarrow  (s', s'+1) \oplus (t', t'+1) \xrightarrow{\gamma} (s',t') \rightarrow 0$$
is a short exact sequence, it is classical that the map $T'_0 =  (s', s'+1) \oplus (t', t'+1)  \oplus \tilde T'_0 \rightarrow (s',t')$ constructed from $\gamma$ and $\tilde f$ admits a kernel. So there is a short exact sequence $$0 \rightarrow T'_1 \rightarrow T'_0 \xrightarrow{f} (s', t') \rightarrow 0$$ where $f$ is a right $\add(T_\sigma)$-approximation. Then, applying the functor $\Hom_{\CM(\Lambda)}(T_\sigma, -)$ to the short exact sequence, and because $T_\sigma$ is cluster tilting (Theorem \ref{thm:cluster tilting}), we obtain easily that $T'_1 \in \add(T_\sigma)$.

Notice now that, for any $(i, j) \in \CM^{\ZZ}(\Lambda)$, if $\Ext^1_{\CM^\ZZ(\Lambda)}((s,i_0), (i,j)) \neq 0$, then $i_0-n < i < j < i_0$ by Theorem \ref{th:graded modules} (2), so we can lift the short exact sequence $$0 \rightarrow T'_1 \rightarrow T'_0 \rightarrow (s', t') \rightarrow 0$$ to a short exact sequence $$0 \rightarrow T_1 \rightarrow T_0 \rightarrow (s, i_0) \rightarrow 0$$
of graded Cohen-Macaulay $\Lambda$-modules where $T_0, T_1 \in \add(T'_{\sigma'})$. So $(s, i_0) \in \thick(T'_{\sigma'})$.

For any non-projective $(s,t) \in \CM^\ZZ(\Lambda)$ satisfying $s < i_0 < t < s+n-1$, there is a short exact sequence
$$0 \rightarrow (t+1-n, i_0) \rightarrow (t+1-n, t) \oplus (s, i_0) \rightarrow (s, t) \rightarrow 0$$
and $(t+1-n, t)$ is projective, so, as $(t+1-n, i_0), (s, i_0) \in \thick(T'_{\sigma'})$, we get $(s, t) \in \thick(T'_{\sigma'})$. In other terms, every $(s, t) \in \CM^\ZZ(\Lambda)$ such that $s < i_0$ and $t \geq i_0$ is in $\thick(T'_{\sigma'})$. By applying $\Omega$, we get that also every $(s, t) \in \CM^\ZZ(\Lambda)$ such that $i_0 - n < s < t \leq i_0$ is in $\thick(T'_{\sigma'})$. Putting both together, every $(s, t)$ such that $i_0-n < s < i_0$ is in $\thick(T'_{\sigma'})$. As applying $\Omega^2$ adds $2-n$ to $s$, we concludes that $\thick(T'_{\sigma'})=\underline{\CM}^{\ZZ}(\Lambda)$.

Therefore, $T'_{\sigma'}$ is a tilting object in $\underline{\CM}^\ZZ(\Lambda)$.

(2) Take the triangulation $\sigma$ whose set of diagonals is $\{(P_1,P_3),(P_1,P_4),\cdots,(P_1,P_{n-1})\}$.
The full subquiver $Q$ of $Q_\sigma$ with the set $Q_{\sigma,0} \smallsetminus F$ of vertices is an alternating quiver of type $A_{n-3}$. 
Thus, we have $$\Gamma_\sigma^{\operatorname{op}} / (e_F) \cong (KQ)^{\operatorname{op}}.$$
For any $i, j$ such that $1<i,j<n$ and $k \in \ZZ$, we have
$$(1,j)(2k) = (1+kn,j+kn) \and (1,j)(2k+1) = (j+kn,1+(k+1)n)$$
so $\End_{\underline{\CM}^ZZ(\Lambda)}((1,i),(1,j)(\ell)) = 0$ for $\ell \neq 0$. Indeed, if $\ell < 0$ this is immediate and if $\ell > 0$, any morphism factors through $(1,n)$ which is projective. Thus $\End_{\underline{\CM}^\ZZ(\Lambda)}(e_F \Gamma_\sigma) \cong \End_{\underline{\CM}(\Lambda)}(e_F \Gamma_\sigma )$. 

By Theorem \ref{th:endomorphism},  we have the following isomorphism 
\begin{align*}
\End_{\underline{\CM}^\ZZ(\Lambda)}(e_F \Gamma_\sigma)& \cong \End_{\underline{\CM}(\Lambda)}(e_F \Gamma_\sigma ) \\
&\cong \Gamma_\sigma^{\operatorname{op}} / (e_F).
\end{align*}

 Because $\underline{\CM}^\ZZ(\Lambda)$ is an algebraic triangulated Krull-Schmidt category, and $e_F \Gamma_\sigma$ is a tilting object in $\underline{\CM}^\ZZ(\Lambda)$, by Theorem \ref{BK}, there exists a triangle-equivalence $$\underline{\CM}^\ZZ(\Lambda) \cong \ck^{\operatorname{b}}(\proj\End^{\operatorname{op}}_{\underline{\CM}^\ZZ(\Lambda)}(e_F \Gamma_\sigma)).$$
Since $\gldim K{Q}_{n-3} < \infty$, we have a triangle-equivalence $$\ck^{\operatorname{b}}(\proj\End^{\operatorname{op}}_{\underline{\CM}^\ZZ(\Lambda)}(e_F \Gamma_\sigma)) \cong\ck^{\operatorname{b}}(\proj KQ) \cong \cd^{\operatorname{b}}(KQ).$$
Therefore there is a triangle-equivalence $\underline{\CM}^\ZZ(\Lambda) \cong \cd^{\operatorname{b}}(KQ)$. 

%
\end{proof}
 
\bibliographystyle{plain}
\bibliography{biblio}
\end{document}